\documentclass[11pt,a4paper,reqno]{amsart}
\usepackage{a4wide}
\usepackage[utf8]{inputenc}
\usepackage[english]{babel}
\usepackage[fixlanguage]{babelbib}
\usepackage{color,graphicx}
\usepackage{amssymb,amsfonts,amsmath,amsthm,amscd, epsfig}
\usepackage[all]{xy}
\usepackage{tensor}
\usepackage{microtype}
\usepackage{tikz-cd}
\usepackage{float}
\usepackage{wasysym}
\usepackage{indentfirst}
\usepackage{tikz}
\usetikzlibrary{cd}

\newcounter{contador}
\numberwithin{contador}{section}

\newtheorem{theorem}[contador]{Theorem}
\newtheorem{prop}[contador]{Proposition}
\newtheorem{lemma}[contador]{Lemma}
\newtheorem{corollary}[contador]{Corollary}

\theoremstyle{definition}
\newtheorem{defi}[contador]{Definition}
\newtheorem{obs}[contador]{Remark}
\newtheorem{exe}[contador]{Example}

\newcommand{\G}{\mathcal{G}}
\newcommand{\tr}{\text{tr}}

\begin{document}

\title{Galois Theory under inverse semigroup actions}
\author[Lautenschlaeger and Tamusiunas]{Wesley G. Lautenschlaeger$^*$ and Thaísa Tamusiunas}
\address{Instituto de Matem\'{a}tica, Universidade Federal do Rio Grande do Sul,  Av. Bento Gon\c{c}alves, 9500, 91509-900. Porto Alegre-RS, Brazil}
\email{wesleyglautenschlaeger@gmail.com}
\email{thaisa.tamusiunas@gmail.com}
\date{}
\thanks{$^*$ Corresponding author}

    
\begin{abstract}
 We develop a Galois theory of commutative rings under actions of finite inverse semigroups. We prove equivalences for the definition of Galois extension as well as a Galois correspondence theorem. We also discuss how the theory behaves in the case of inverse semigroups with zero. 
\end{abstract}

\maketitle

\vspace{0.5 cm}



\section{Introduction}

S. U. Chase, D. K. Harrison, and A. Rosenberg (CHR) developed, in 1965, a Galois Theory for finite groups acting on commutative rings, where they presented a generalization of the Fundamental Theorem of Galois Theory \cite{chase1969galois}. This work was grounded in the concept of Galois extensions of commutative rings as introduced by M. Auslander and O. Goldman in \cite{auslander1960brauer}. Shortly thereafter, Villamayor and Zelinski (VZ) also developed a commutative Galois theory in \cite{villamayor1966galois}, a little more general than CHR.  In their approach, they considered a Galois extension $R \subset S$ and a finite group $G$ of $R$-automorphisms of $S$, and established a correspondence between separable $R$-subalgebras of $S$ and certain class of subgroups of $G$ (the so-called ``fat'' ones). The tool to make this correspondence was a groupoid acting on a determined ring, and this was the first time that a groupoid action appeared in the context of Galois theory. Given an action $\beta = (A_g,\beta_g)_{g \in \G}$ of a groupoid $\G$ on a ring $A$, $\beta$ is said to be \emph{orthogonal} if $A = \bigoplus_{e \in \G_0} A_e$, where $\G_0$ is the set of identities of $\G$. In VZ, the groupoid action was orthogonal, as well as in other subsequent works that concern Galois theory for groupoid actions (e.g. \cite{bagio2012partial}, \cite{basata2021galois}, \cite{cortes2017characterisation}, \cite{garta2024}, \cite{paques2018galois}, \cite{pataIII} and \cite{pedrotti2023injectivity}), and it was required in order to ensure the invariance of the trace map.

Until then, the study of Galois theory for groupoid actions was limited to non-ordered groupoids. In this scenario, in \cite{lata2021galois} a Galois theory for ordered groupoids was introduced, and then the Ehresmann-Schein-Nambooripad (ESN) Theorem \cite[Theorem 4.1.8]{lawson1998inverse}, which establishes an equivalence between the categories of inverse semigroups and of inductive groupoids, was used to translate the Galois theory for orthogonal groupoid actions to the case of orthogonal inverse semigroup actions. However, the orthogonality condition for inverse semigroup actions requires that $A = \bigoplus_{e \in \max E(S)} A_e$, where $E(S)$ is the set of idempotent elements of the inverse semigroup $S$, and it trivializes ideals that are associated with elements that are not maximal with respect to the natural partial order of $S$. Therefore, the class of actions considered in \cite{lata2021galois} is not the most general possible.

The idea of relying on the ready-made theory of groupoids to construct a theory for inverse semigroups turns out not being the most effective way. Starting from this, we took a completely different approach to what had been done so far. The main idea was to relate $E$-unitary inverse semigroup actions to partial group actions to construct a new invariant trace map and, with this, obtain a Galois theory for $E$-unitary inverse semigroups; after that, extend the theory to any inverse semigroup action.

Therefore, the goal of this paper is to develop a Galois theory for inverse semigroup actions without requiring the orthogonality condition. The paper is structured as follows. In Section 2, we review fundamental concepts regarding inverse semigroups and Galois theory. Sections 3 and 4 focus on constructing a trace map that is invariant and proving an equivalence theorem for Galois extensions for the case of $E$-unitary inverse semigroups. Section 5 is focused on the Galois correspondence. Precisely, we demonstrate that if $S$ is a finite $E$-unitary inverse semigroup acting injectively on a commutative ring $A$ via a unital action $\beta$, then there is a one-to-one correspondence between the $\beta$-complete inverse subsemigroups of $S$ and the separable, $\beta$-strong $A^{\beta}$-subalgebras of $A$, where $A^\beta$ is the subring of invariants of $A$ over $\beta$. In Section 6, we present how to apply the results from the previous sections to construct a Galois correspondence theorem for general inverse semigroup actions. Finally, in Section 7, we discuss a theory involving inverse semigroups with zero. 

Throughout, rings and algebras are associative.

\section{Basic results}

This section is destined to the prerequisites concerning inverse semigroups, partial actions of groups and Galois extensions. 

\subsection{Inverse semigroups and partial orders}

The main reference of this subsection is \cite{lawson1998inverse}. A semigroup $S$ is called \emph{inverse} if for each $s \in S$ there exists a unique element, denoted $s^{-1}$, such that the following two equations are satisfied:
\begin{align*}
    ss^{-1}s = s \quad \text{and} \quad s^{-1}ss^{-1} = s^{-1}.
\end{align*}

An \emph{inverse subsemigroup} of an inverse semigroup is a subsemigroup closed under inverses.  An inverse semigroup with identity is called an \emph{inverse monoid}. We have that $(s^{-1})^{-1} = s$ and the elements $s^{-1}s$ and $ss^{-1}$ are idempotents, for all $s \in S$, where an \emph{idempotent} in a semigroup is an element $e$ such that $e^2 = e$. The set of all idempotents of the inverse semigroup $S$ is denoted by $E(S)$. An inverse subsemigroup $T$ of $S$ is called \emph{full} if $E(S) \subseteq T$.

Let $S$ be a semigroup equipped with a partial order $\leq$. We say
that $S$ is a \emph{partially ordered semigroup} if $a \leq b$ and $c \leq d$ imply that $ac \leq bd$.

Now consider an inverse semigroup $S$ and the following partial order on $S$:
\begin{align}\label{paror}
    s \preceq t \iff s = tf, \text{ for some } f \in E(S).
\end{align} 

This partial order is called the \emph{natural partial order on} $S$. An inverse semigroup $S$ is $E$\emph{-unitary} if, whenever $e \in E(S)$ and $e \preceq s$, then $s \in E(S)$, for all $s \in S$. 

Let $P$ be a set with a partial order $\leq$. We say that the subset $Q \subseteq P$ is an \emph{order ideal} of $P$ if given elements $p \in P$, $q \in Q$ such that $p \leq q$, then $p \in Q$. It is well known that $E(S)$ is an order ideal of an inverse semigroup $S$ with the natural partial order. We say that an element $w$ in $P$ is a \emph{meet} of the elements $a,b \in P$ if $w \leq a, b$ and if $c \leq a, b$ then $c \leq w$. In this case, we use the notation $w = a \wedge b$.

Let $S$ be an inverse semigroup. For $a, b \in S$, define $a \sim b$ if and only if $a^{-1}b, ab^{-1} \in E(S)$. This is called the \emph{compatibility relation}. Two elements $a, b \in S$ are said to be \emph{compatible} if $a \sim b$. Also, we define the relation $\sigma$ on $S$ by
    \begin{align*}
        s \sigma t \iff \exists u \preceq s,t,
    \end{align*}
    for all $s,t \in S$.

We say that an element $u \in S$ is a \emph{join} of the elements $s,t \in S$ if $s,t \preceq u$ and, moreover, if $s,t \preceq v$, then $u \preceq v$. In this case, we denote $u = s \vee t$.  An inverse semigroup is called  \emph{complete} if every non-empty compatible subset has a join. 

Next, we introduce a weaker notion of completeness. 
    
\begin{defi}
An inverse semigroup is said to be \emph{$f$-complete} if every non-empty \textbf{finite} compatible subset has a join.
\end{defi}

Clearly, if $S$ is complete, then it is $f$-complete.

     An inverse subsemigroup $T$ of $S$ is \emph{closed under joins} if, for any compatible subset $A \subseteq T$ such that $\bigvee A$ exists, it follows that $\bigvee A \in T$.

    A \emph{congruence} on a semigroup $S$ is an equivalence relation $\rho$ on $S$ such that $(a,b), (c,d) \in \rho$ implies $(ac, bd) \in \rho$.

\begin{theorem} \cite[Theorem 2.4.1]{lawson1998inverse} \label{teoconggrp1}
    Let $S$ be an inverse semigroup.
    \begin{enumerate}
        \item[(i)] $\sigma$ is the smallest congruence in $S$ containing $\sim$.

        \item[(ii)] $S/\sigma$ is a group.

        \item[(iii)] If $\rho$ is a congruence in $S$ such that $S/\rho$ is a group, then $\sigma \subseteq \rho$.
    \end{enumerate}
\end{theorem} 

This is why we refer to  $\sigma$ as the \emph{minimum group congruence} in $S$. For any $s \in S$, we denote $\sigma(s)$ as the congruence class of the element $s$ under the relation $\sigma$. These relations and notations will be maintained throughout the paper.

\begin{theorem} \cite[Theorem 2.4.6]{lawson1998inverse} \label{teosigmaigsim}
    Let $S$ be an inverse semigroup. Then the following conditions
are equivalent:
    \begin{enumerate}
        \item[(i)] $S$ is $E$-unitary.

        \item[(ii)] $\sim = \sigma$.

        \item[(iii)] $\sigma(e) = E(S)$, for any $e \in E(S)$.
    \end{enumerate}
\end{theorem}

 Now assume that $R$ is a commutative ring and that $A$ is an $R$-algebra. Let $B \subseteq A$ a subset. We denote $B \triangleleft_u A$ to indicate that $B$ is a unital (bilateral) ideal of $A$ with identity $1_B$, which is a central idempotent in $A$; that is, $B = A1_B = 1_BA$. We define the set:  \begin{center}$\text{Iso}_{pu}(A) = \{ f : I \to J \mid I,J \triangleleft_u A \text{ and } f \text{ is an isomorphism of unital rings} \}$.\end{center} This set forms an inverse semigroup with composition between maps defined as follows:
 \begin{align*}
    fg = f|_{\text{im}(g) \cap \text{dom}(f)} \circ g|_{g^{-1}(\text{im}(g) \cap \text{dom}(f))}.
 \end{align*}

\begin{lemma}\label{lem-compatib}
Let $f, g \in \emph{\text{Iso}}_{pu}(A)$. Then $f \sim g$ if and only if
\begin{align*}
    f|_{\emph{\text{dom}}(f) \cap \emph{\text{dom}}(g)} = g|_{\emph{\text{dom}}(f) \cap \emph{\text{dom}}(g)} \text{ and } f^{-1}|_{\emph{\text{im}}(f) \cap \emph{\text{im}}(g)} = g^{-1}|_{\emph{\text{im}}(f) \cap \emph{\text{im}}(g)}.
\end{align*}
\end{lemma}
\begin{proof}
    Assume that $f|_{\text{dom}(f) \cap \text{dom}(g)} = g|_{\text{dom}(f) \cap \text{dom}(g)} \text{ and } f^{-1}|_{\text{im}(f) \cap \text{im}(g)} = g^{-1}|_{\text{im}(f) \cap \text{im}(g)}.$
    Given arbitrary \(x \in \text{dom} \, f \cap \text{dom} \, g\) and \(y \in \text{im} \, f \cap \text{im} \, g\), we have:  
\[
f^{-1}g(x) = f^{-1}f(x) = x \quad \text{and} \quad fg^{-1}(y) = ff^{-1}(y) = y.
\]

Therefore, \(f^{-1}g\) and \(fg^{-1}\) are identity functions, i.e., idempotent. Consequently, \(f \sim g\). Conversely, if \(f \sim g\), then \(f^{-1}g\) and \(fg^{-1}\) are idempotent in \(\text{Iso}_{pu}(A)\), meaning they are identity functions. Thus,  
\[
f^{-1}g(x) = x, \,\, \text{for all} \,\, x \in \text{dom}(f^{-1}g) = g^{-1}(\text{dom}(f^{-1}) \cap \text{im}(g)) = g^{-1}(\text{im}(f) \cap \text{im}(g)),
\]
and  
\[
fg^{-1}(y) = y, \,\, \text{for all} \,\, y \in \text{dom}(fg^{-1}) = g(\text{dom}(f) \cap \text{im}(g^{-1})) = g(\text{dom}(f) \cap \text{dom}( g)).
\]

Then, if \(x \in \text{dom}(f^{-1}g)\), there exists \(x' \in \text{im}(f) \cap \text{im}(g)\) such that \(g^{-1}(x') = x\). Therefore,  
\[
g^{-1}(x') = x = f^{-1}g(x) = f^{-1}g(g^{-1}(x')) = f^{-1}(x'),
\]
which means that $f^{-1}|_{\text{im}(f) \cap \text{im}(g)} = g^{-1}|_{\text{im}(f) \cap \text{im}(g)}.$

Similarly, if \(y \in \text{dom}(fg^{-1})\), there exists \(y' \in \text{dom} \, f \cap \text{dom} \, g\) such that \(g(y') = y\). Hence,  
\[
g(y') = y = fg^{-1}(y) = fg^{-1}(g(y')) = f(y'),
\]
which implies that  $f|_{\text{dom}(f) \cap \text{dom}(g)} = g|_{\text{dom}(f) \cap \text{dom}(g)}$.
\end{proof}

Using the above lemma, we can prove that  $\text{Iso}_{pu}(A)$ is $f$-complete.

\begin{prop}
    $\emph{\text{Iso}}_{pu}(A)$ is $f$-complete.
\end{prop}
\begin{proof}
 Let $f: I_f \to J_f$ and $g: I_g \to J_g$ in $\text{Iso}_{pu}(A)$ such that $f \sim g$. Define the map $f + g : I_f + I_g \to J_f + J_g$ by $(f + g)(x + y) = f(x) + g(y)$, for $x \in I_f$ and $y \in I_g$. Observe that $f + g$ is well-defined because, by Lemma \ref{lem-compatib},  \(f \sim g\) implies that \(f\) and \(g\) coincide on \(I_f \cap I_g\), and \(f^{-1}\) and \(g^{-1}\) coincide on \(J_f \cap J_g\). Hence, we have $f^{-1} + g^{-1} = (f + g)^{-1}$, and $f + g$ is a bijection. Next, we will demonstrate that $f + g$ is a ring isomorphism. The additivity is straightforward, so we will focus on proving that $f + g$ preserves multiplication. Let $x_1, x_2 \in I_f$ and $y_1,y_2 \in I_g$. Then:
{\small\begin{align*}
(f+g)((x_1+y_1)(x_2+y_2)) & = (f+g)(x_1x_2 + x_1y_2 + y_1x_2 + y_1y_2)\\
& = f(x_1x_2 + x_1y_2) + g(y_1x_2 + y_1y_2)\\
& = f(x_1)f(x_2) + f(x_1y_2)f(1_{I_f}1_{I_g}) + g(y_1x_2)g(1_{I_f}1_{I_g}) + g(y_1)g(y_2)\\
& = f(x_1)f(x_2) + f(x_1)f(y_21_{I_f}1_{I_g}) + g(y_1)g(x_21_{I_f}1_{I_g}) + g(y_1)g(y_2)\\
& = f(x_1)f(x_2) + f(x_1)g(y_21_{I_f}1_{I_g}) + g(y_1)f(x_21_{I_f}1_{I_g}) + g(y_1)g(y_2)\\
& = f(x_1)f(x_2) + f(x_1)g(y_2)g(1_{I_f}1_{I_g}) + g(y_1)f(x_2)f(1_{I_f}1_{I_g}) + g(y_1)g(y_2)\\
& = f(x_1)f(x_2) + f(x_1)g(y_2)1_{J_f}1_{J_g} + g(y_1)f(x_2)1_{J_f}1_{J_g} + g(y_1)g(y_2)\\
& = f(x_1)f(x_2) + f(x_1)g(y_2) + g(y_1)f(x_2) + g(y_1)g(y_2)\\
& = (f+g)(x_1+y_1)(f+g)(x_2+y_2).
\end{align*}}


Also, the finite sum of unital ideals remains a unital ideal, concluding that $f+g \in \text{Iso}_{pu}(A)$. Furthermore, it is straightforward to verify that $f \vee g = f + g$.
 
 This construction can be extended to a finite number of compatible elements in $\text{Iso}_{pu}(A)$ as follows: Given isomorphisms $f: I_f \to J_f$, $g: I_g \to J_g$, and $h : I_h \to J_h$ in $\text{Iso}_{pu}(A)$ such that $f,g,h$ are $\sim$-related, we will show that $h$ is $\sim$-related to $f + g$. Notice that $\text{dom}(f+g) = I_f + I_g$, hence $I_h \cap \text{dom}(f+g) = I_h \cap (I_f + I_g)$. Since these ideals are unital, we have that $I_h \cap (I_f + I_g) = I_h(I_f + I_g) = I_hI_f + I_hI_g = (I_h \cap I_f) + (I_h \cap I_g)$. Consider $x + y \in (I_h \cap I_f) + (I_h \cap I_g)$, where $x \in I_h \cap I_f$ and $y \in I_h \cap I_g$. Then $h(x) = f(x)$, since $h \sim f$, and $h(y) = g(y)$, since $h \sim g$. Thus, we have:
 \begin{align*}
    h(x+y) = h(x) + h(y) = f(x) + g(y) = (f+g)(x+y),
 \end{align*}
proving that $h$ and $f + g$ coincide on the intersection of their domains. Similarly, we prove that \(h^{-1}\) and \(f^{-1} + g^{-1}\) coincide on the intersection of their domains. Consequently, by Lemma \ref{lem-compatib}, $h \sim (f + g)$. By symmetry,  we can also conclude that  $f \sim g + h$ and $g \sim f + h$, ensuring that $f + g + h$ is well-defined. The statement then follows by induction.
\end{proof}

Let $S$ be a finite $E$-unitary inverse subsemigroup of $\text{Iso}_{pu}(A)$. In an $E$-unitary semigroup, the compatibility relation coincides with the minimum group congruence, as established in Theorem \ref{teosigmaigsim}. Let $s \in S$. Then the finite set $\sigma(s)$ consists of pairwise compatible elements in $S$. Once $\text{Iso}_{pu}(A)$ is $f$-complete, we can define the element $\alpha_s = \sum_{g \in \sigma(s)} g \in \text{Iso}_{pu}(A)$. Let $G' = \{\alpha_s : s \in S\} \subseteq \text{Iso}_{pu}(A)$.

We adapt \cite[Lemmas 7.2.1, 7.2.2]{lawson1998inverse} for $G' = \{\alpha_s : s \in S\} \subseteq \text{Iso}_{pu}(A)$.

\begin{lemma} \label{lemaparceunit1}
Following the notations defined above, we have:
    \begin{enumerate}
        \item[(i)] For all $\alpha, \beta \in G'$, there exists a unique $\gamma_{(\alpha,\beta)} \in G'$ such that $\alpha \beta \preceq \gamma_{(\alpha,\beta)}$.

        \item[(ii)] The element $\alpha_e$ is idempotent in $G'$, for any $e \in E(S)$.

        \item[(iii)] If $\alpha \in G'$, then $\alpha^{-1} \in G'$.
    \end{enumerate}
\end{lemma}
\begin{proof}
    (i): Firstly, observe that $f \preceq g$ in $\text{Iso}_{pu}(A)$ is equivalent to $f$ being a restriction of $g$ in the usual sense of functions.

    Let $\alpha = \alpha_f$ and $\beta = \alpha_g$ for some $f,g \in S$. Define $\gamma_{(\alpha,\beta)} = \alpha_{fg}$. We will prove that $\text{dom}(\alpha_f\alpha_g) \subseteq \text{dom}(\alpha_{fg})$ and that $\alpha_f\alpha_g(x) = \alpha_{fg}(x)$, for all $x \in \text{dom}(\alpha_f\alpha_g)$.  To this end, we denote $\{f_1, \ldots, f_n\} = \sigma(f)$ and $\{g_1, \ldots, g_m\} = \sigma(g)$.
    
    If $x \in \text{dom}(\alpha_f\alpha_g) = \alpha_g^{-1}(\text{dom}(\alpha_f) \cap \text{im}(\alpha_g))$, then $x \in \text{dom} (\alpha_g)$. This implies that $x = \sum_{i=1}^m x_i$, with $x_i \in \text{dom}(g_i)$, for $1 \leq i \leq m$. Moreover, $y = \alpha_g(x) = \sum_{i=1}^m g_i(x_i)$ is an element of dom$(\alpha_f)$. We can rewrite this sum if necessary, and we can assume that 
$y$ is expressed as $y = \sum_{i=1}^n y_k$ where $y_k \in \text{dom}(f_k)$ and \(y_k = \sum_{i \in I_k} g_i(x_i)\) for some index set \(I_k \subseteq \{1, \ldots, m\}\).  Thus, we have: \begin{align*}
        \alpha_f(y) & = \sum_{k=1}^n f_k(y_k) = \sum_{k=1}^n f_k \left( \sum_{i \in I_k} g_i(x_i) \right ) = \sum_{k=1}^n f_k \left( \alpha_g \left ( \sum_{i \in I_k} x_i \right) \right).
        \end{align*} 

    
    Therefore, for each $1 \leq k \leq n$, $\sum_{i \in I_k} g_i(x_i)$ is an element of dom$(f_k)$. This means that $\sum_{i \in I_k} x_i$ is also an element of the domain of $\alpha_{fg}$ for all $1 \leq i \leq m$, since it is an element of the domain of $f_k\alpha_g$. Consequently, $x \in \text{dom}(\alpha_{fg})$. Now, since $\sigma(fg) = \sigma(f)\sigma(g)$, it is easy to see that $\alpha_{fg}(x)$ coincides with $\alpha_f \alpha_g (x)$. So $\alpha_{f}\alpha_g \preceq \alpha_{fg}$.

   For the uniqueness, assume that there is $\gamma' \in G'$ such that $\alpha_f\alpha_g \preceq \gamma'$. By the construction of $G'$, we have that $\gamma' = \alpha_h$, for some $h \in S$. Then $\alpha_f\alpha_g \preceq \gamma' = \alpha_h$ implies $fg \preceq \alpha_h$. Since $h \preceq \alpha_h$, we have $(fg) \sigma h$. Thus, we conclude that $\alpha_h = \alpha_{fg} = \gamma_{(\alpha,\beta)}$.

The proofs of parts (ii) and (iii) are analogous to the argument presented in \cite[Lemma 7.2.1]{lawson1998inverse}.
\end{proof}

Define a binary operation $\astrosun$: $G' \times G' \rightarrow G'$ by $\alpha \astrosun \beta = \gamma_{(\alpha,\beta)}$, where $\gamma_{(\alpha,\beta)}$ is the unique element in $G'$  that is greater than or equal to $\alpha \beta$, as established in Lemma \ref{lemaparceunit1}(i). Observe that $\alpha_e = \sum_{g \in \sigma(e)} g = \sum_{g \in E(S)} g$, where the last equality follows from Theorem \ref{teosigmaigsim}(iii). Consequently, if $e, f \in E(S)$, then $\alpha_e = \alpha_f$.

\begin{lemma} \label{lemaparceunit2}
    $(G',\astrosun)$ is a group isomorphic to $S / \sigma$; the identity is $\alpha_e$, $e \in E(S)$, and the inverse of $\alpha \in G'$ is the partial bijection $\alpha^{-1} \in G'$. 
\end{lemma}
\begin{proof}
    Define $\Phi : (G', \astrosun) \to S / \sigma$ by $\Phi(\alpha_f) = \sigma(f)$. By Lemma \ref{lemaparceunit1}(i), we obtain that $\Phi(\alpha_f \astrosun \alpha_g) = \Phi(\alpha_{fg}) = \sigma(fg) = \sigma(f)\sigma(g) = \Phi(\alpha_f)\Phi(\alpha_g)$. Additionally, by Lemma \ref{lemaparceunit1}(iii), we have that $\Phi(\alpha_f^{-1}) = \Phi(\alpha_{f^{-1}}) = \sigma(f^{-1}) = \sigma(f)^{-1} = \Phi(\alpha_f)^{-1}$. Furthermore, for any $e \in E(S)$, it follows from Lemma \ref{lemaparceunit1}(ii) that $\Phi(\alpha_e) = 1_{S/\sigma}$. Moreover, $\Phi$ is evidently a bijection, as the elements of $(G',\astrosun)$ are defined by the $\sigma$-classes of $S$.
\end{proof}

\subsection{Actions of inverse semigroups and partial actions of groups} In this subsection, we present the definitions of an inverse semigroup action on a ring and a partial group action on a ring.

\begin{defi} Let $S$ be an inverse semigroup and $A$ be a ring.  An \emph{action of $S$ on} $A$ is a  collection $\beta = (A_s,\beta_s)_{s \in S}$ such that $A_s \triangleleft A$, $\beta_s : A_{s^{-1}} \to A_s$ is an isomorphism of $R$-algebras, for all $s \in S$, and
\begin{itemize}
\item[(i)] $A = \sum_{e \in E(S)} A_e$;
\item[(ii)] $\beta_s(\beta_t(a)) = \beta_{st}(a)$, for all $s,t \in S$ and $a \in A_{(st)^{-1}}=\beta_t^{-1}(A_{t} \cap A_{s^{-1}})$.
\end{itemize} \end{defi}


The definition of an inverse semigroup action on a ring provided by Exel and Vieira in \cite[Definition 2.3]{exel2010actions} does not require that $A = \sum_{e \in E(S)} A_e$. However, Beuter, Gonçalves, Öinert, and Royer in \cite[Definition 2.1]{beuter2019simplicity} introduced the concept of a partial action of an inverse semigroup on a ring, incorporating the condition $A = \sum_{e \in E(S)} A_e$. This condition is natural as it considers rings “filled” by the ideals of the action. So, we adopt it in our paper. Also, it is not difficult to see that every global action of $S$ on $A$, as defined above, is a partial action of $S$ on $A$ (in the sense of Beuter et al).

\begin{obs}
From \cite[Proposition 2.2]{beuter2019simplicity}, it follows that $\beta_e = \text{id}_{A_e}$, for any $e \in A_e$, and if $S$ is an inverse semigroup with unit $1_S$, then $A_{1_S} = A$.
\end{obs}

An action is called \emph{injective} if $\beta_{s_1} = \beta_{s_2}$ implies $s_1 = s_2$, for all $s_1, s_2 \in S$. A \emph{unital action} of $S$ on $A$ is an action $\beta$ of $S$ on $A$ such that every ideal $A_s$, for $s \in S$, is generated by a central idempotent of $A$, denoted by $1_s$. Thus, $A_s = 1_sA = A1_s$. 

Observe that a unital action of  $S$ on $A$ can be interpreted as an inverse semigroup homomorphism $\beta : S \to \text{Iso}_{pu}(A)$ such that $A = \sum_{e \in E(S)} \text{im}(\beta(e))$. To establish this, define $\beta(s) := \beta_s $, $\text{dom}(\beta(s))  := A_{s^{-1}}$ and $\text{im}(\beta(s)) :=  A_{s}$, for each $s \in S$.

Next, we introduce the notion of partial action of a group on a ring. 

\begin{defi} \label{pagroup} \cite[p. 1932]{dokuchaev2005associativity} Let $G$ be a group and $A$ be a ring. A \emph{partial action of $G$ on $A$} is a collection $\alpha = (A_g,\alpha_g)_{g \in G}$ of ideals $A_g$ of $A$ and ring isomorphisms $\alpha_g : A_{g^{-1}} \to A_g$, $g \in G$, such that:
\begin{enumerate}
    \item[(P1)] $A_{1_G} = A$ and $\alpha_{1_G} = \text{Id}_A$;

    \item[(P2)] $\alpha_h^{-1}(A_{g^{-1}} \cap A_h) \subseteq A_{(gh)^{-1}}$, for all $g,h \in G$;

    \item[(P3)] $\alpha_g \circ \alpha_h(x) = \alpha_{gh}(x)$, for all $g,h \in G$, $x \in \alpha_h^{-1}(A_{g^{-1}} \cap A_h)$.
\end{enumerate}
\end{defi}

When the ideals are unital, we say that $\alpha$ is a \emph{unital partial action}.

\subsection{Galois extensions}
Let $S$ be a finite inverse semigroup acting on a ring $A$ via a unital action $\beta = (A_s,\beta_s)_{s \in S}$. We set
\begin{align*}
    A^\beta = \{a \in A : \beta_s(a1_{s^{-1}}) = a1_s, \text{ for all } s \in S\},
\end{align*}
the subring of invariants of $A$ by $\beta$. Notice that $A = \sum_{e \in E(S)} A_e$ implies that $A$ is a unital ring, with identity named $1_A$.

Let $A\supseteq B$ be a ring extension. The ring $A$ is said to be $B$-separable if there exists an element $e = \sum_{i = 1}^nx_i \otimes_{B} y_i \in A \otimes_{B} A$ such that $\sum_{i = 1}^nx_iy_i = 1_A$ and $(1_A \otimes a)e = (a \otimes 1_A)e$ for all $a \in A$ (c.f. \cite{HS}). This element $e$ is referred to as \emph{idempotent of separability of} $A$ over $A^{\beta}$. When $B$ is commutative, it is equivalent to state that $A$ is a projective ($A \otimes_{B} A^{op}$)-module, where $A^{op}$ denotes the opposite algebra of $A$ (c.f. \cite{chase1969galois}).

We say that $A$ is a $\beta$-Galois extension of $A^\beta$ if there exist $\{x_i,y_i\}_{i=1}^n$ elements of $A$ such that
    \begin{align*}
        \sum_{i=1}^n x_i\beta_s(y_i1_{s^{-1}}) = \sum_{e \in E(S)} 1_e\delta_{e,s}.
    \end{align*}
    
In this case, we say that $\{ x_i, y_i \}_{i=1}^n$ is a Galois coordinate system for the extension $A|_{A^\beta}$ (c.f. \cite{lata2021galois}).

We shall prove in Section 5 that there exists a one-to-one correspondence between the $A^\beta$-separable and $\beta$-strong subalgebras of $A$ and determined inverse subsemigroups of $S$. Thus, we need to define the concept of $\beta$-strong subalgebra. Let $B$ be an $A^\beta$-subalgebra of $A$. We set:
    \begin{align*}
        S_B = \{ s \in S : \beta_s(b1_{s^{-1}}) = b1_s, \text{ for all } b \in B \}.
    \end{align*}

\begin{defi}Let $B$ be an $A^\beta$-subalgebra of $A$. We say that $B$ is $\beta$-\emph{strong} if for all $s,t \in S$, with $s^{-1}t \notin S_B$, and for all nonzero idempotent $e \in A_s \cup A_t$, there exists an element $b \in B$ such that $\beta_s(b1_{s^{-1}})e \neq \beta_t(b1_{t^{-1}})e$.
\end{defi}

One of our main interests is to use partial actions of groups to solve the problem of the invariance of the trace map. To this end, we conclude this subsection by introducing the definition of a Galois extension for partial group actions.

Given a unital partial action $\alpha = (A_g,\alpha_g)_{g \in G}$ of the group $G$ on the ring $A$, we set $A^{\alpha} = \{a \in A : \alpha_g(a1_{g^{-1}}) = a1_g, \text{ for all } g \in G\}$, that is, the subring of invariants of $A$ by $\alpha$.

The ring $A$ is said to be a partial Galois extension of $A^{\alpha}$, or an $\alpha$-Galois extension of $A^{\alpha}$, if there exist elements $\{x_i,y_i\}_{i=1}^n$ in $A$ such that
$\sum_{i=1}^n x_i\alpha_g(y_i1_{g^{-1}}) = \delta_{1_{G},g},$ for each $g \in G$ (c.f. \cite{dokuchaev2007partial}).

\section{The trace map}
For the remainder of this paper, we assume that $S$ is a \textbf{finite} inverse semigroup. Let $\beta = (A_s,\beta_s)_{s \in S}$ be a unital action of $S$ on the ring $A$. We define the trace map $\tr_\beta : A \to A$ as $$\tr_\beta(a) = \sum_{s \in S} \beta_s(a1_{s^{-1}}).$$

If $S$ is, in particular, a group, it is easy to see that the trace map is invariant under $\beta$, that is, $\tr_\beta(A) \subseteq A^\beta$. However, this property does not necessarily hold for actions of inverse semigroups, as we illustrate in the following example.

\begin{exe}
Consider the inverse monoid $S = \{1_S, s, s^{-1}, ss^{-1}, s^{-1}s, t, tt^{-1}\}$ with relations $t = stt^{-1}$, $t = s^{-1}tt^{-1}$, $t = t^{-1}$ and $s^2 = s^{-2} = tt^{-1}$.

Let $A = \mathbb{C}e_1 \oplus \mathbb{C}e_2 \oplus \mathbb{C}e_3$, where $\mathbb{C}$ is the set of complex numbers and $e_1, e_2, e_3$ are central idempotents pairwise orthogonal such that $e_1 + e_2 + e_3 = 1_A$. Define $A_{1_S} = A$, $A_{s^{-1}} = \mathbb{C}e_1 \oplus \mathbb{C}e_2$, $A_s = \mathbb{C}e_2 \oplus \mathbb{C}e_3$, $A_t = \mathbb{C}e_2$. Also define $\beta_s(ae_1 + be_2) = \overline{b}e_2 + ae_3$, $\beta_{s^{-1}}(be_2 + ce_3) = ce_1 + \overline{b}e_2$, $\beta_t(be_2) = \overline{b}e_2$ and $\beta_e = \text{Id}_{A_e}$, for all $e \in E(S)$, where $\overline{b}$ indicates the complex conjugate of $b$. Then, $A^\beta = \mathbb{C}(e_1 + e_3) \oplus \mathbb{R}e_2$. However, $\tr_\beta(ie_2) = ie_2 \notin A^\beta$, where $i$ is the imaginary number. 
\end{exe}

To establish a bijective Galois correspondence, we need of the invariance of the trace map. Therefore, we must define another ``trace map" in such way that it is invariant. We shall work with (finite, as stated at the beginning of this section) $E$-unitary inverse semigroups. 

\begin{theorem} \label{teoconstracparcsigma}
    Let $S$ be a $E$-unitary inverse semigroup acting on $A$ via $\beta = (A_s,\beta_s)_{s \in S}$. Assume that $\beta$ is a unital injective action. Define
    \begin{align} \label{defacparc1}
        A_{\sigma(s)} = \sum_{t \in \sigma(s)} A_t
    \end{align}
    and
    \begin{align} \label{defacparc2}
        \alpha_{\sigma(s)} = \sum_{t \in \sigma(s)} \beta_t.
    \end{align}
Then $\alpha = (A_{g},\alpha_g)_{g \in G}$ is a unital partial action of the group $G = S / \sigma$ on $A$.
\end{theorem}
\begin{proof}
    By hypothesis, $\beta : S \to \text{Iso}_{pu}(A)$, $s \mapsto \beta_s$, is injective, so $S$ is isomorphic to its image $\beta(S)$ in $\text{Iso}_{pu}(A)$. Clearly, $\beta(S)$ is $E$-unitary. We will identify $S$ with its image $\beta(S)$ in $\text{Iso}_{pu}(A)$.
    
Let $G' = \{ \alpha_{\sigma(s)} : s \in S\} \subseteq \text{Iso}_{pu}(A)$. By Lemma \ref{lemaparceunit2}, $G'$ is a group with the operation $\astrosun$ and $(G',\astrosun) \simeq \beta(S) / \sigma \simeq S / \sigma = G$. Now, take $\alpha = (A_g,\alpha_g)_{g \in G}$ as defined in \eqref{defacparc1} and \eqref{defacparc2}. We will verify that $\alpha$ is a partial action of $G$ on $A$. To this end, observe that for $e \in E(S)$, $A_{\sigma(e)}= \sum_{f \in\sigma(e)} A_f \overset{\text{\tiny{Th.}} \ref{teosigmaigsim}(iii)}{=} \sum_{f \in E(S)} A_f = A$. Furthermore, $\alpha_{\sigma(e)} = \text{Id}_{A_{\sigma(e)}} = \text{Id}_A$. This proves (P1) of Definition \ref{pagroup}. 
    
     By Lemma \ref{lemaparceunit1}, $\alpha_g \circ \alpha_h \preceq \alpha_{gh}$. Then:
     \begin{align*}
         \alpha_h^{-1}(A_{g^{-1}} \cap A_h) = \text{dom}(\alpha_g \circ \alpha_h) \subseteq \text{dom}(\alpha_{gh}) = A_{(gh)^{-1}}
     \end{align*}
     and
     \begin{align*}
         \alpha_g \circ \alpha_h = \alpha_{gh}|_{\text{dom}(\alpha_g \circ \alpha_h)} = \alpha_{gh}|_{\alpha_h^{-1}(A_{g^{-1}} \cap A_h)},
     \end{align*}
     which are precisely (P2) and (P3) of Definition \ref{pagroup}.
     
     Let $s \in S$ and assume $\sigma(s) = \{s_1, \ldots, s_m\}$. Then, $A_{\sigma(s)} = \sum_{t \in \sigma(s)} A_t = \sum_{i=1}^m A_{s_i}$. Thus $A_{\sigma(s)}$ is a unital ring with identity given by the boolean sum 
    \begin{align*}
        1_{\sigma(s)} = \sum_{k=1}^m \sum_{i_1 \leq \cdots \leq i_k} (-1)^{k+1} 1_{s_{i_1}} \cdots 1_{s_{i_k}}.
    \end{align*}
\end{proof}

\begin{prop}
   Using the notations and assumptions of Theorem \ref{teoconstracparcsigma}, we have $A^\beta = A^{\alpha}$.
\end{prop}
\begin{proof}
    Let $a \in A^\beta$, that is, $a \in A$ is such that $\beta_s(a1_{s^{-1}}) = a1_s$, for all $s \in S$. Thus:
    
    \begin{align*}
        \alpha_{\sigma(s)}(a1_{\sigma(s)^{-1}}) & = \alpha_{\sigma(s)}\left ( a \sum_{k=1}^n \sum_{i_1 \leq \cdots \leq i_k} (-1)^{k+1} 1_{s_{i_1}^{-1}} \cdots 1_{s_{i_k}^{-1}} \right ) \\
        & = \sum_{k=1}^n \sum_{i_1 \leq \cdots \leq i_k} (-1)^{k+1} \alpha_{\sigma(s)}(a1_{s_{i_1}^{-1}}) \cdots \alpha_{\sigma(s)}(1_{s_{i_k}^{-1}}) \\
        & = \sum_{k=1}^n \sum_{i_1 \leq \cdots \leq i_k} (-1)^{k+1} \beta_{s_{i_1}}(a1_{s_{i_1}^{-1}}) \cdots \beta_{s_{i_k}}(1_{s_{i_k}^{-1}})
        \end{align*}

        \begin{align*}
        & = \sum_{k=1}^n \sum_{i_1 \leq \cdots \leq i_k} (-1)^{k+1} a1_{s_{i_1}} \cdots 1_{s_{i_k}} \\
        & = a\sum_{k=1}^n \sum_{i_1 \leq \cdots \leq i_k} (-1)^{k+1} 1_{s_{i_1}} \cdots 1_{s_{i_k}} = a1_{\sigma(s)},
    \end{align*}
    where $\sigma(s) = \{s_1, \ldots, s_n\}$. Since $\sigma(s) \in S / \sigma$ is arbitrary, it follows that $a \in A^\alpha$.

    Reciprocally, suppose that $a \in A^\alpha$, that is, $\alpha_{\sigma(s)}(a1_{\sigma(s)^{-1}}) = a1_{\sigma(s)}$, for all $\sigma(s) \in S / \sigma$. Then: 
    \begin{align*}
        \beta_s(a1_{s^{-1}}) & = \alpha_{\sigma(s)}(a1_{s^{-1}}) = \alpha_{\sigma(s)}(a1_{\sigma(s)^{-1}}1_{s^{-1}}) = \alpha_{\sigma(s)}(a1_{\sigma(s)^{-1}})\alpha_{\sigma(s)}(1_{s^{-1}}) \\
        & = a1_{\sigma(s)}\beta_s(1_{s^{-1}}) = a1_{\sigma(s)}1_s = a1_s,
    \end{align*}
    for all $s \in S$. Therefore $a \in A^\beta$.
\end{proof}

Now, we define the $\sigma$-trace map $\tr^{\sigma}_\beta := \tr_\alpha: A \rightarrow A$ given by $\tr^{\sigma}_\beta(a) = \tr_\alpha(a) = \sum_{g \in S/\sigma} \alpha_g(a1_{g^{-1}})$. 

\begin{corollary}
Under the assumptions of Theorem \ref{teoconstracparcsigma}, the map $\tr^{\sigma}_\beta$ is a homomorphism of $A^\beta$-bimodules. Furthermore, $\tr^{\sigma}_\beta(A) \subseteq A^\beta$ and $\tr^{\sigma}_\beta(\beta_s(a)) = \tr^{\sigma}_\beta(a)$, for all $s \in S$ and $a \in A_{s^{-1}}$.
\end{corollary}
\begin{proof}
    The map $\tr^{\sigma}_\beta$ is a homomorphism of $A^\beta$-bimodules and $\tr^{\sigma}_\beta(A) \subseteq A^\beta$ from \cite[Lemma 2.1(i)]{dokuchaev2007partial}. For the last claim, let $s \in S$ and $a \in A_{s^{-1}}$. Then:
    \begin{align*}
        \tr^{\sigma}_\beta(\beta_s(a)) & = \tr_\alpha(\beta_s(a))  = \sum_{g \in S / \sigma} \alpha_g(\beta_s(a)1_{g^{-1}}) = \sum_{g \in S / \sigma} \alpha_g(\alpha_{\sigma(s)}(a)1_{g^{-1}}) \\
        & = \sum_{g \in S / \sigma} \alpha_{g\sigma(s)}(a1_{(g\sigma(s))^{-1}})1_g  = \sum_{h \in S / \sigma} \alpha_{h}(a1_{h^{-1}})1_{h\sigma(s)^{-1}} \\
        & = \sum_{h \in S / \sigma} \alpha_{h}(a1_{h^{-1}})\alpha_h(1_{h^{-1}}1_{\sigma(s)^{-1}})  = \sum_{h \in S / \sigma} \alpha_{h}(a1_{h^{-1}}1_{\sigma(s)^{-1}}) \\
        & = \sum_{h \in S / \sigma} \alpha_{h}(a1_{h^{-1}}) = \tr_\alpha(a) = \tr^{\sigma}_\beta(a).
    \end{align*}
\end{proof}

Assume the notations and assumptions of Theorem \ref{teoconstracparcsigma}. The ring  $A$ is a $(A^\beta,A *_\alpha G)$-bimodule (respectively $(A *_\alpha G,A^\beta)$-bimodule) with actions of $A^\beta$ on the left and on the right given by multiplication and action of $A *_{\alpha} G$ on the right given by:
\begin{align*}
    a \cdot b_g\delta_g = \alpha_{g^{-1}}(ab_g),
\end{align*}
and on the left by:
\begin{align*}
    b_g\delta_g \cdot a = b_g\alpha_g(a1_{g^{-1}}),
\end{align*}
for all $g \in G$, $b_g \in A_g$, $a \in A$ and linearly extended.

The map
\begin{align*}
    T : A \times A & \to A^\beta \\
        (a,b) & \mapsto \tr^{\sigma}_\beta(ab)
\end{align*}
is $A *_\alpha G$-balanced and the map
\begin{align*}
    T' : A \times A & \to A *_\alpha G \\
    (a,b) & \mapsto \sum_{g \in G} a\alpha_g(b1_{g^{-1}})\delta_g
\end{align*}
is $A^{\beta}$-balanced.

Therefore, we can define the maps
\begin{align*}
    \tau : A \otimes_{A *_\alpha G} A & \to A^\beta \\
        a \otimes b & \mapsto \tr^{\sigma}_\beta(ab)
\end{align*}
and
\begin{align*}
    \tau' : A \otimes_{A^\beta} A & \to A *_\alpha G \\
        a \otimes b & \mapsto \sum_{g \in G} a\alpha_g(b1_{g^{-1}})\delta_g.
\end{align*}

\begin{prop} 
The sextuple $(A *_\alpha G, A^\beta, A, A, \tau, \tau')$ is a Morita context. The map $\tau$ is surjective if and only if there exists $a \in A$ such that $\tr^{\sigma}_\beta(a) = 1_A$.
\end{prop}
\begin{proof}
It follows by \cite[Proposition 4.4]{bagio2012partial} applied on the partial action $\alpha$ of $G$ on $A$.
\end{proof}

\section{Galois extensions}

For what follows in this section, assume that $S$ is a $E$-unitary inverse semigroup acting injectively on $A$ via a unital action $\beta = (A_s,\beta_s)_{s \in S}$, and $\alpha = (A_{g},\alpha_g)_{g \in G}$ is the partial action of the group $G = S / \sigma$ on $A$ as defined in Theorem \ref{teoconstracparcsigma}.

\begin{prop} \label{propgaloissse}
   $A$ is a $\beta$-Galois extension of $A^\beta$ if and only if $A$ is a  $\alpha$-Galois extension of $A^\alpha$.
\end{prop}
\begin{proof}
    Suppose that $A$ is a $\beta$-Galois extension of $A^\beta = A^\alpha$. So there exist $\{ x_i, y_i\}_{i=1}^n$ in $A$ such that
    \begin{align*}
        \sum_{i=1}^n x_i\beta_s(y_i1_{s^{-1}}) = \sum_{e \in E(S)} 1_e\delta_{e,s}.
    \end{align*}

    Given an element $g \in G$,
    \begin{align*}
        \sum_{i=1}^n x_i\alpha_g(y_i1_{g^{-1}}) & = \sum_{i=1}^n x_i\alpha_g\left (y_i \sum_{k = 1}^m\sum_{i_1 \leq \cdots \leq i_k} (-1)^{k+1} 1_{s_{i_1}^{-1}} \cdots 1_{s_{i_k}^{-1}} \right ) \\
        & = \sum_{i=1}^n \sum_{k = 1}^m (-1)^{k+1} \sum_{i_1 \leq \cdots \leq i_k}  x_i\alpha_g(y_i 1_{s_{i_1}^{-1}} \cdots 1_{s_{i_k}^{-1}}) \\
        & = \sum_{k = 1}^m (-1)^{k+1} \sum_{i_1 \leq \cdots \leq i_k}  \left ( \sum_{i=1}^n x_i\alpha_g(y_i 1_{s_{i_1}^{-1}}) \right )\alpha_g(1_{s_{i_2}}^{-1}) \cdots \alpha_g(1_{s_{i_k}^{-1}}))
        \end{align*}

        \begin{align*}
        & = \sum_{k = 1}^m (-1)^{k+1} \sum_{i_1 \leq \cdots \leq i_k}  \left ( \sum_{i=1}^n x_i\beta_{s_{i_1}}(y_i 1_{s_{i_1}^{-1}}) \right )\beta_{s_{i_2}}(1_{s_{i_2}}^{-1}) \cdots \beta_{s_{i_k}}(1_{s_{i_k}^{-1}})) \\
        & = \sum_{k = 1}^m (-1)^{k+1} \sum_{i_1 \leq \cdots \leq i_k}  \left ( \sum_{e \in E(S)} 1_e \delta_{e,s_{i_1}} \right )1_{s_{i_2}}\cdots 1_{s_{i_k}},
    \end{align*}
    where $g = \{ s_1, \ldots, s_m\}$. If $g \neq 1_{G}$, then $s_i \notin E(S)$, for all $1 \leq i \leq m$. Hence $\sum_{i=1}^n x_i\alpha_g(y_i1_{g^{-1}})$ $= 0.$

    On the other hand, if $g = 1_{G}$, then $s_i \in E(S)$, for all $1 \leq i \leq m$. Thus:
    \begin{align*}
        \sum_{i=1}^n x_i\alpha_g(y_i1_{g^{-1}}) & = \sum_{k = 1}^m (-1)^{k+1} \sum_{i_1 \leq \cdots \leq i_k}  \left ( \sum_{e \in E(S)} 1_e \delta_{e,s_{i_1}} \right )1_{s_{i_2}}\cdots 1_{s_{i_k}} \\
        & = \sum_{k = 1}^m (-1)^{k+1} \sum_{i_1 \leq \cdots \leq i_k}  1_{s_{i_1}}1_{s_{i_2}}\cdots 1_{s_{i_k}} \\
        & = \sum_{k = 1}^m \sum_{i_1 \leq \cdots \leq i_k} (-1)^{k+1} 1_{s_{i_1}}1_{s_{i_2}}\cdots 1_{s_{i_k}} \\
        & = 1_A = 1_{g}.
    \end{align*}

    Therefore, $\sum_{i=1}^n x_i\alpha_g(y_i1_{g^{-1}}) = 1_A\delta_{1_{G},g},$ for all $g \in G$, that is, $A$ is a $\alpha$-Galois extension of $A^\alpha$.

    Reciprocally, suppose that $A$ is a $\alpha$-Galois extension of $A^\alpha$ and consider a Galois coordinate system $\{ x_i, y_i\}_{i=1}^n$ of $A$ on $A^\alpha$. Hence, for all $g \in G$, $\sum_{i=1}^n x_i\alpha_g(y_i1_{g^{-1}}) = \delta_{1_{G},g}.$

    Let $s \in S$. Then
    \begin{align*}
        \sum_{i=1}^n x_i\beta_s(y_i1_{s^{-1}}) & = \sum_{i=1}^n x_i\alpha_{\sigma(s)}(y_i1_{s^{-1}}) = \sum_{i=1}^n x_i\alpha_{\sigma(s)}(y_i1_{\sigma(s)^{-1}}1_{s^{-1}}) \\
        & = \sum_{i=1}^n x_i\alpha_{\sigma(s)}(y_i1_{\sigma(s)^{-1}})\alpha_{\sigma(s)}(1_{s^{-1}})  = \sum_{i=1}^n x_i\alpha_{\sigma(s)}(y_i1_{\sigma(s)^{-1}})\beta_{s}(1_{s^{-1}}) \\
        & = \sum_{i=1}^n x_i\alpha_{\sigma(s)}(y_i1_{\sigma(s)^{-1}})1_{s}  = \left ( \sum_{i=1}^n x_i\alpha_{\sigma(s)}(y_i1_{\sigma(s)^{-1}}) \right )1_{s} \\
        & = 1_A \delta_{\sigma(s),1_{G}}1_s  = 1_s \delta_{\sigma(s),1_{G}}.
    \end{align*}

    If $s \in E(S)$, then $\sigma(s) = 1_{G}$, hence $\sum_{i=1}^n x_i\beta_s(y_i1_{s^{-1}}) = 1_s.$ On the other hand, if $s \notin E(S)$, then $\sum_{i=1}^n x_i\beta_s(y_i1_{s^{-1}}) = 0.$ Thus
    \begin{align*}
        \sum_{i=1}^n x_i\beta_s(y_i1_{s^{-1}}) & = \sum_{e \in E(S)} 1_e\delta_{e,s},
    \end{align*}
    that is, $A$ is $\beta$-Galois extension of $A^\beta$.
\end{proof}

Consider the map $j : A *_\alpha G \to End(A)_{A^\beta}$ given by $j\left (\sum_{g \in G} a_g \delta_g \right )(a) = \sum_{g \in G} a_g\alpha_g(a1_{g^{-1}}).$ It is  clearly a well-defined homomorphism of $A$-modules and of rings.

Let $M$ be a left $A *_\alpha G$-module. We set:
\begin{align*}
    M^S = \{m \in M : 1_s\delta_{\sigma(s)} \cdot m = 1_sm, \text{ for all } s \in S\},
\end{align*}
the $A^\beta$-submodule of invariants of $M$ by $S$. Note that $M$ is a left $A$-module via the embedding $a \mapsto a1_{A *_\alpha G}$ of $A$ in $A *_\alpha G$. We also set:
\begin{align*}
    M^{G} = \{m \in M : 1_g\delta_{g} \cdot m = 1_gm, \text{ for all } g \in G\}.
\end{align*}

\begin{lemma} \label{lemaigm}
    Following the notations above, $M^S = M^{G}$.
\end{lemma}
\begin{proof}
    If $m \in M^{G}$, then
    \begin{align*}
        1_s\delta_{\sigma(s)} \cdot m = 1_s \cdot (1_{\sigma(s)}\delta_{\sigma(s)} \cdot m) = 1_s \cdot (1_{\sigma(s)} \cdot m) = (1_s \cdot 1_{\sigma(s)}) \cdot m = 1_s \cdot m,
    \end{align*}
    for all $s \in S$. Thus $m \in M^S$.

    Reciprocally, if $m \in M^S$, then, taking $\{ s_1, \ldots s_n \} = g \in G$,
    \begin{align*}
        1_g \delta_{g} \cdot m & = \left ( \sum_{k=1}^n\sum_{i_1 \leq \cdots \leq i_k} (-1)^{k+1} 1_{s_{i_1}} \cdots 1_{s_{i_k}} \delta_{\sigma(s_{i_k})} \right ) \cdot m \\
        & = \sum_{k=1}^n\sum_{i_1 \leq \cdots \leq i_k}  \left ( (-1)^{k+1} 1_{s_{i_1}} \cdots 1_{s_{i_{k-1}}} \right ) \cdot  (1_{s_{i_k}} \delta_{\sigma(s_{i_k})} \cdot m) \\
        & = \sum_{k=1}^n\sum_{i_1 \leq \cdots \leq i_k}  \left ( (-1)^{k+1} 1_{s_{i_1}} \cdots 1_{s_{i_{k-1}}} \right ) \cdot  (1_{s_{i_k}} \cdot m) \\
        & = \sum_{k=1}^n\sum_{i_1 \leq \cdots \leq i_k}  \left ( (-1)^{k+1} 1_{s_{i_1}} \cdots  1_{s_{i_k}} \right ) \cdot m
     = 1_g \cdot m.
    \end{align*} 
\end{proof}

\begin{defi}
    Consider the ring $A_\beta(S) = \prod_{s \in S} A_s$. We define the subset $PA_\beta(S)$ of $A_\beta(S)$ as
    \begin{align*}
        PA_\beta(S) = \{ (a_s)_{s \in S} : \text{ if } s \preceq t, \text{ then } a_s = a_t1_s \}.
    \end{align*}

    Notice that $PA_\beta(S)$ is a subring of $A_\beta(S)$ and it is also a left $A$-submodule of $A_\beta(S)$.
\end{defi}

We shall present a Galois correspondence for an inverse semigroup acting on a commutative ring. So, for the remainder of this  section, assume that the ring $A$ is commutative. Most of the following characterization for Galois extensions is a consequence from Theorem \ref{teoconstracparcsigma}, Proposition \ref{propgaloissse}, Lemma \ref{lemaigm} and the references \cite{bagio2012partial} and \cite{dokuchaev2007partial}.

\begin{theorem} \label{teoequivseminvgal}
    The following statements are equivalent:
    \begin{enumerate}
        \item[(i)] $A$ is a $\beta$-Galois extension of $A^\beta$;

        \item[(ii)] $A$ is a finitely generated projective $A^\beta$-module and $j$ is an isomorphism of left $A$-modules and of rings.

        \item[(iii)] For every left $A *_\alpha G$-module $M$, the map $\mu : A \otimes M^S \to M$ given by $\mu(a \otimes m) = am$ is an isomorphism of left $A$-modules.

        \item[(iv)] The map $\psi : A \otimes_{A^\beta} A \to PA_\beta(S)$ defined by $\psi(a \otimes b) = (a\beta_s(b1_{s^{-1}}))_{s \in S}$ is an isomorphism of left $A$-modules.

        \item[(v)] $A$ is $A^\beta$-separable and $\beta$-strong.

        \item[(vi)] $AtA = A *_{\alpha} G$, where $t = \sum_{g \in G} 1_g\delta_g$.
    
        \item[(vii)] The map $\tau'$ is surjective.
    
        \item[(viii)] $A$ is a generator for the category of left $A *_\alpha G$-modules.
    
        \item[(ix)] $\tr^\sigma_{\beta}(A) = A^\beta$.
    
        \item[(x)] $A$ is a generator for the category of right $A^\beta$-modules.
    
        \item[(xi)] The Morita context $(A *_\alpha G, A^\beta, A, A, \tau, \tau')$ is strict.
    \end{enumerate}
\end{theorem}
\begin{proof}
    By Proposition \ref{propgaloissse}, $A$ is a $\beta$-Galois extension of $A^\beta$ if and only if $A$ is a $\alpha$-Galois extension of $A^\alpha = A^\beta$. Therefore, by \cite[Theorem 4.1]{dokuchaev2007partial} and by Lemma \ref{lemaigm}, (i) $\Leftrightarrow$ (ii) $\Leftrightarrow$ (iii). Moreover, by \cite[Theorem 5.3, Corollary 5.4]{bagio2012partial}, (i) $\Leftrightarrow$ (vi) $\Leftrightarrow$ (vii) $\Leftrightarrow$ (viii) $\Leftrightarrow$ (ix) $\Leftrightarrow$ (x) $\Leftrightarrow$ (xi). Now we shall prove the missing equivalences.

    (iii) $\Rightarrow$ (iv): By \cite[Theorem 4.1]{dokuchaev2007partial}, $A \otimes_{A^\beta} A \stackrel{\rho}{\simeq} \prod_{g \in G} A_g$ via $a \otimes b \mapsto \rho(a \otimes b) = (a\alpha_g(b1_{g^{-1}}))_{g \in G}$. We will prove that $\prod_{g \in G} A_g \simeq PA_\beta(S)$. Define
    \begin{align*}
        \varphi : \prod_{g \in G} A_g & \to PA_\beta(S) \\
        (a_g)_{g \in G} & \mapsto (a_{\sigma(s)}1_s)_{s \in S}.
    \end{align*}

    Note that $\varphi$ is well-defined, because if $s \preceq t$, then $\sigma(s) = \sigma(t)$, and so $a_s = a_{\sigma(s)}1_s = a_{\sigma(t)}1_s = a_{\sigma(t)}1_t1_s = a_t1_s$. Furthermore, it is easy to see that $\varphi$ is a homomorphism of $A$-modules. We shall see that $\varphi$ is a bijection.

    Let $(a_g)_{g \in G}, (b_g)_{g \in G} \in \prod_{g \in G} A_g$ such that $\varphi((a_g)_{g \in G}) = \varphi((b_g)_{g \in G})$. Thus $a_{\sigma(s)}1_t = b_{\sigma(s)}1_t$, for all $t \in \sigma(s)$. Hence
    \begin{align*}
        a_{\sigma(s)} & = a_{\sigma(s)}1_{\sigma(s)} = a_{\sigma(s)}\sum_{k=1}^m \sum_{i_1 \leq \cdots \leq i_k} (-1)^{k+1}1_{s_{i_1}} \cdots 1_{s_{i_k}} \\
        & = \sum_{k=1}^m \sum_{i_1 \leq \cdots \leq i_k} (-1)^{k+1}a_{\sigma(s)}1_{s_{i_1}} \cdots 1_{s_{i_k}} = \sum_{k=1}^m \sum_{i_1 \leq \cdots \leq i_k} (-1)^{k+1}b_{\sigma(s)}1_{s_{i_1}} \cdots 1_{s_{i_k}} \\
        & = b_{\sigma(s)}\sum_{k=1}^m \sum_{i_1 \leq \cdots \leq i_k} (-1)^{k+1}1_{s_{i_1}} \cdots 1_{s_{i_k}} = b_{\sigma(s)}1_{\sigma(s)} = b_{\sigma(s)},
    \end{align*}
    where $\{s_1, \ldots, s_m\} = \sigma(s)$, for all $s \in S$. Therefore, $(a_g)_{g \in G} = (b_g)_{g \in G}$.

    For the surjectivity, let $(a_s)_{s \in S} \in PA_\beta(S)$. Consider, for all $s \in S$,
    \begin{align*}
        a_{\sigma(s)} & = \sum_{k=1}^m \sum_{i_1 \leq \cdots \leq i_k} (-1)^{k+1} 1_{s_{i_1}} \cdots 1_{s_{i_k}}a_{s_{i_k}},
    \end{align*}
    where $\sigma(s) = \{s_1, \ldots, s_m\}$. Then
    \begin{align*}
        \varphi((a_g)_{g \in G}) & = (a_{\sigma(s)}1_s)_{s \in S} = \left ( \sum_{k=1}^m \sum_{i_1 \leq \cdots \leq i_k} (-1)^{k+1} 1_{s_{i_1}} \cdots 1_{s_{i_k}}a_{s_{i_k}}1_s \right )_{s \in S} \\
        & \stackrel{(*)}{=} \left ( 1_sa_s \right )_{s \in S} = (a_s)_{s \in S}.
    \end{align*}
    Now, we shall prove ($*$). Let $a = (-1)^{k+1}1_{s_{i_1}} \cdots 1_{s_{i_{k}}}a_{s_{i_k}}$ be a term in the sum above, with $1 \leq k \leq m$. Then $a = (-1)^{k+1}1_{s_{i_1} \wedge \cdots \wedge s_{i_k}} a_{s_{i_1} \wedge \cdots \wedge s_{i_k}}$, by $\sigma = \sim$ and the definition of $PA_\beta(S)$. We have two possibilities: either $s \in \{s_{i_1}, \ldots, s_{i_k}\}$ or $s \notin \{s_{i_1}, \ldots, s_{i_k}\}$.
    
    If $s \notin \{ s_{i_1}, \ldots, s_{i_k} \}$ (in particular $k < m$ and if $k=1$, then $a$ is not the factor $1_sa_s$), thus $a1_s = (-1)^{k+1}1_{s_{i_1} \wedge \cdots \wedge s_{i_k} \wedge s} a_{s_{i_1} \wedge \cdots \wedge s_{i_k} \wedge s}$. Then $\{ s_{i_1}, \ldots, s_{i_k}, s \}$ is a subset of $k+1$ distinct elements of $\sigma(s)$. Therefore, $(-1)^{k+2} 1_{s_{i_1} \wedge \cdots \wedge s_{i_k} \wedge s} a_{s_{i_1} \wedge \cdots \wedge s_{i_k} \wedge s} = -a1_s$ also appears in the sum, canceling the element $a1_s$.

    If $s \in \{ s_{i_1}, \ldots, s_{i_k} \}$, then $a1_s = a$. Suppose that $k \neq 1$. Suppose also, without loss of generality, that $s_{i_k} = s$. Consider the element $b = (-1)^{k}1_{s_{i_1} \wedge \cdots \wedge s_{i_{k-1}}} a_{s_{i_1} \wedge \cdots \wedge s_{i_{k-1}}}$. Hence $b1_s = -a1_s$. Therefore these elements are also canceled in the sum.

    The only element that does not cancel is the element $a_s1_s = a_s$, and ($*$) is proved.

    Furthermore, $\varphi \circ \rho = \psi$. Indeed,
    \begin{align*}
        \varphi \circ \rho (a \otimes b) & = \varphi((a\alpha_{g}(b1_{g^{-1}}))_{g \in G}) 
         = (a\alpha_{\sigma(s)}(b1_{\sigma(s)^{-1}})1_s)_{s \in S} \\
        & = (a\alpha_{\sigma(s)}(b1_{\sigma(s)^{-1}})\beta_{s}(1_{s^{-1}}))_{s \in S} 
         = (a\alpha_{\sigma(s)}(b1_{\sigma(s)^{-1}})\alpha_{\sigma(s)}(1_{s^{-1}}))_{s \in S} \\
        & = (a\alpha_{\sigma(s)}(b1_{\sigma(s)^{-1}}1_{s^{-1}}))_{s \in S} 
         = (a\alpha_{\sigma(s)}(b1_{s^{-1}}))_{s \in S} \\
        & = (a\beta_{s}(b1_{s^{-1}}))_{s \in S} 
         = \psi(a \otimes b).
    \end{align*}

    (iv) $\Rightarrow$ (i): We have $\left (\sum_{e \in E(S)} \delta_{e,s} 1_s\right)_{s \in S} \in PA_\beta(S)$, because if $e \leq f$ in $E(S)$, then $1_e = 1_e1_f$ and no $s \in S \setminus E(S)$ satisfies $e \leq s$, for $e \in E(S)$. Since $\psi$ is an isomorphism, there exists $\sum_{i=1}^n x_i \otimes y_i \in A \otimes_{A^\beta} A$ such that
    \begin{align*}
        \left ( \sum_{e \in E(S)} \delta_{e,s}1_s \right )_{s \in S} = \psi \left ( \sum_{i=1}^n x_i \otimes y_i \right ) = \left ( \sum_{i=1}^n x_i\beta_s(y_i1_{s^{-1}}) \right )_{s \in S}.
    \end{align*}

    (i) $\Rightarrow$ (v): Consider a Galois coordinate system $\{x_i,y_i\}_{i=1}^n$ of $A$ over $A^\beta$. Let $s = \sum_{i=1}^n x_i \otimes y_i \in A \otimes_{A^\beta} A$. We shall prove that $s$ is the idempotent of separability of $A$ over $A^\beta$. Indeed, for all $e \in E(S)$, 
    \begin{align*}
        1_A = 1_{\sigma(e)} = \sum_{i=1}^n x_i\alpha_{\sigma(e)}(y_i1_{\sigma(e)}) = \sum_{i=1}^n x_iy_i1_{\sigma(e)} = \sum_{i=1}^n x_iy_i.
    \end{align*}

\noindent \textbf{Claim:} for all $1 \leq i \leq n$ and $a \in A$, $ax_i = \sum_{j=1}^n x_j\tr^\sigma_{\beta}(y_jax_i)$. 

In fact,
    \begin{align*}
        \sum_{j=1}^n x_j\tr^\sigma_{\beta}(y_jax_i) & = \sum_{j=1}^n x_j \left ( \sum_{g \in G} \alpha_g(y_jax_i1_{g^{-1}}) \right ) 
         = \sum_{g \in G} \left (  \sum_{j=1}^n x_j\alpha_g(y_j1_{g^{-1}}) \right )\alpha_g(ax_i1_{g^{-1}}) \\
        & = 1_A\alpha_{1_G}(ax_i1_{1_G}) = ax_i.
    \end{align*}

    Thus, for all $a \in A$,
    \begin{align*}
        \sum_{i=1}^n ax_i \otimes y_i & = \sum_{i=1}^n\sum_{j=1}^n x_j\tr_\beta^\sigma(y_jax_i) \otimes y_i 
         = \sum_{i=1}^n\sum_{j=1}^n x_j \otimes \tr_\beta^\sigma(y_jax_i)y_i \\
        & = \sum_{j=1}^n x_j \otimes \sum_{i=1}^n\sum_{g \in G} \alpha_g(y_ja1_{g^{-1}})\alpha_g(x_i1_{g^{-1}})y_i \\
        & = \sum_{j=1}^n x_j \otimes \sum_{g \in G} \alpha_g(y_ja1_{g^{-1}})\left ( \sum_{i=1}^n \alpha_g(x_i1_{g^{-1}})y_i \right )
        \end{align*}

        \begin{align*}
        & = \sum_{j=1}^n x_j \otimes \alpha_{1_G}(y_ja1_{1_G})1_A 
         = \sum_{j=1}^n x_j \otimes y_ja.
    \end{align*}

    Therefore, $A$ is $A^\beta$-separable. To show that $A$ is $\beta$-strong, let $s,t \in S$ such that $s^{-1}t \notin E(S)$. Without loss of generality, consider $e \in E(A_s)$. If $\beta_s(a1_{s^{-1}})e = \beta_t(a1_{t^{-1}})e$, for all $a \in A$, then
    \begin{align*}
        a\beta_{s^{-1}}(e) = \beta_{s^{-1}}(\beta_{t}(a1_{t^{-1}})e) = \beta_{s^{-1}t}(a1_{t^{-1}s})\beta_{s^{-1}}(e),
    \end{align*}
    for all $a \in A$. In particular,
    \begin{align*}
        \beta_{s^{-1}}(e) & = 1_A\beta_{s^{-1}}(e) 
         = \sum_{i=1}^n x_iy_i\beta_{s^{-1}}(e) 
         = \sum_{i=1}^n x_i\beta_{s^{-1}t}(y_i1_{t^{-1}s})\beta_{s^{-1}}(e) 
         = 0 \cdot \beta_{s^{-1}}(e) = 0,
    \end{align*}
    since $s^{-1}t \notin E(S)$. Thus, $\beta_{s^{-1}}(e) = 0$, and once $\beta_{s^{-1}}$ is an isomorphism, $e = 0$. Thus $A$ is $\beta$-strong.
    
    (v) $\Rightarrow$ (i): Let $e = \sum_{i=1}^n x_i \otimes y_i \in A \otimes_{A^\beta} A$ the idempotent of separability of $A$ over $A^\beta$. That is, $m(e) = \sum_{i=1}^n x_iy_i = 1_A$ and $(a \otimes 1_A - 1_A \otimes a)e = 0$, for all $a \in A$. For each $s \in S$, define $e_s = m((\text{Id}_A \otimes \beta_s(-1_{s^{-1}}))(e)) \in A_s$. We have that $\text{Id}_A \otimes \beta_s(-1_{s^{-1}})$ is as $(A \otimes_{A^\beta} A)$-endomorphism. Moreover, since $A$ is commutative, $m$ is a ring homomorphism. Thus, for each $s \in S$, 
    \begin{align*}
        e_s^2 & = m((\text{Id}_A \otimes \beta_s(-1_{s^{-1}}))(e)) \cdot m((\text{Id}_A \otimes \beta_s(-1_{s^{-1}}))(e)) \\
        & = m((\text{Id}_A \otimes \beta_s(-1_{s^{-1}}))(e) \cdot (\text{Id}_A \otimes \beta_s(-1_{s^{-1}}))(e)) \\
        & = m((\text{Id}_A \otimes \beta_s(-1_{s^{-1}}))(e^2)) = m((\text{Id}_A \otimes \beta_s(-1_{s^{-1}}))(e)) = e_s.
    \end{align*}

    That is, $e_s$ is an idempotent of $A_s$. On the other hand, for all $a \in A$, 
    \begin{align*}
        \beta_{ss^{-1}}(a1_{ss^{-1}})e_s & = ae_s = am((\text{Id}_A \otimes \beta_s(-1_{s^{-1}}))(e)) = am((\text{Id}_A \otimes \beta_s(-1_{s^{-1}}))(e))1_s \\
        & = (a \otimes 1_s) \cdot m((\text{Id}_A \otimes \beta_s(-1_{s^{-1}}))(e)) = m((a \otimes 1_s) \cdot (\text{Id}_A \otimes \beta_s(-1_{s^{-1}}))(e)) \\
        & = m((\text{Id}_A \otimes \beta_s(-1_{s^{-1}}))(a \otimes 1_A) \cdot (\text{Id}_A \otimes \beta_s(-1_{s^{-1}}))(e)) \\
        & = m((\text{Id}_A \otimes \beta_s(-1_{s^{-1}}))((a \otimes 1_A)e)) = m((\text{Id}_A \otimes \beta_s(-1_{s^{-1}}))((1_A \otimes a)e)) \\
        & = m((\text{Id}_A \otimes \beta_s(-1_{s^{-1}}))(1_A \otimes a) \cdot (\text{Id}_A \otimes \beta_s(-1_{s^{-1}}))(e)) \\
        & = m((1_A \otimes \beta_s(a1_{s^{-1}})) \cdot (\text{Id}_A \otimes \beta_s(-1_{s^{-1}}))(e)) \\
        & = m((1_A \otimes \beta_s(a1_{s^{-1}}))) \cdot m((\text{Id}_A \otimes \beta_s(-1_{s^{-1}}))(e)) \\
        & = \beta_s(a1_{s^{-1}})m((\text{Id}_A \otimes \beta_s(-1_{s^{-1}}))(e)) 
         = \beta_s(a1_{s^{-1}})e_s.
    \end{align*}

    Since $A$ is $\beta$-strong, if $s \notin E(S)$, then $e_s = 0$. But
    \begin{align*}
        e_s & = m((\text{Id}_A \otimes \beta_s(-1_{s^{-1}}))(e)) 
         = \sum_{i=1}^n x_i\beta_s(y_i1_{s^{-1}}).
    \end{align*}

    On the other hand, if $s \in E(S)$, then $s = ss^{-1}$ and
    \begin{align*}
        e_s & = m((\text{Id}_A \otimes \beta_s(-1_{s^{-1}}))(e)) 
         = \sum_{i=1}^n x_i\beta_s(y_i1_{s^{-1}}) 
         = \sum_{i=1}^n x_iy_i1_s 
         = 1_A1_s = 1_s.
    \end{align*}

    Therefore, $\{x_i,y_i\}_{i=1}^n$ is a Galois coordinate system of $A$ over $A^\beta$.
\end{proof}

\section{Galois correspondence}

Throughout this section, assume that $A$ is a commutative ring and $S$ is an $E$-unitary inverse semigroup acting injectively on $A$ via a unital action $\beta = (A_s,\beta_s)_{s \in S}$.  Recall that the finiteness of
$S$ was previously established.

Assume that $T$ is an inverse subsemigroup of $S$ and that $P \subseteq T$ is a compatible subset such that the join $u = \bigvee P$ exists. Since $\{\beta_s : s \in P\}$ is a compatible subset of $\text{Iso}_{pu}(A)$, the element $\sum_{s \in P} \beta_s$ is defined. For each $s \in P$, we have $\beta_s \preceq \beta_u$ since $s \preceq u$. By the definition of the join, this implies that $\sum_{s \in P} \beta_s \preceq \beta_u$, but the equality does not necessarily holds.

\begin{defi}
    Let $T$ be an inverse subsemigroup of $S$. We say that $T$ is \emph{$\beta$-complete} if: 
    \begin{itemize} 
        \item[(i)] $T$ is full;  
        \item[(ii)] for any compatible subset $P \subseteq T$ such that $u = \bigvee P$ exists and such that $\beta_u = \sum_{s \in P} \beta_s$, it follows that $u \in T$.
    \end{itemize}
\end{defi}

When $S$ is a finite group, every subgroup of $S$ is $\beta$-complete, because in this case the natural partial order is the equality.

Now recall that $S_B = \{ s \in S : \beta_s(b1_{s^{-1}}) = b1_s, \text{ for all } b \in B \}$, as defined in Section 2.3.
    
\begin{prop}
    Let $B$ be an $A^\beta$-subalgebra of $A$. Then $S_B$ is a $\beta$-complete inverse subsemigroup of $S$.
\end{prop}
\begin{proof}
    We start by showing that $S_B$ is an order ideal. Let $s \in S_B$ and $u \preceq s$. Hence:
    \begin{align*}
        \beta_u(b1_{u^{-1}}) & = \beta_u(b1_{u^{-1}}1_{s^{-1}}) = \beta_u(b1_{s^{-1}})\beta_u(1_{u^{-1}}) 
         = \beta_{s}(b1_{s^{-1}})1_u = b1_s1_u = b1_u,
    \end{align*}
    for all $b \in B$. Thus $u \in S_B$.

    Let $s,t \in S_B$. To prove that $S_B$ is an inverse subsemigroup of $S$, we shall verify that $st,s^{-1} \in S$. By \cite[Theorem 3.1.2(3)]{lawson1998inverse}, there are $u,v \in S$ such that $u \preceq s$, $v \preceq t$, $st = uv$ and $u^{-1}u = vv^{-1}$. But $u,v \in S_B$, since $S_B$ is an order ideal. Then
    \begin{align*}
        \beta_{st}(b1_{(st)^{-1}}) & = \beta_{uv}(b1_{(uv)^{-1}}) = \beta_u(\beta_v(b1_{v^{-1}})1_{u^{-1}}) \\
        & = \beta_u(b1_v1_{u^{-1}}) = \beta_u(b1_{vv^{-1}}1_{u^{-1}u}) \\
        & = \beta_u(b1_{u^{-1}u}) = \beta_u(b1_{u^{-1}}) \\
        & = b1_u = b1_{uu^{-1}} = b1_{u(u^{-1}u)u^{-1}} \\
        & = b1_{u(vv^{-1})u^{-1}} = b1_{(uv)(uv)^{-1}} = b1_{uv} = b1_{st},
    \end{align*}
    for all $b \in B$.

    Moreover, for each $b \in B$, $\beta_{s^{-1}}(b1_s) = \beta_{s^{-1}}(\beta_s(b1_{s^{-1}})) = b1_{s^{-1}}$. We also have that $S_B$ is full, because given elements $e \in E(S)$ and $b \in B$, $\beta_e(b1_{e^{-1}}) = \beta_e(b1_e) = \text{Id}_{A_e}(b1_e) = b1_e.$

    Finally, let $P \subseteq S_B$ such that there is $u = \bigvee P$ and $\beta_u = \sum_{s \in P} \beta_s$. Let $P = \{s_1, \ldots, s_n\}$. Given an element $b \in B$,
   \begin{align*}
        \beta_u(b1_{u^{-1}}) & = \beta_u\left(b\left( \sum_{1 \leq j \leq n}\sum_{i_1 \leq \cdots \leq i_j} (-1)^{j+1}1_{s_{i_1}^{-1}} \cdots 1_{s_{i_j}^{-1}} \right)\right) \\
        & = \beta_u\left(\sum_{1 \leq j \leq n}\sum_{i_1 \leq \cdots \leq i_j} b(-1)^{j+1}1_{s_{i_1}^{-1}} \cdots 1_{s_{i_j}^{-1}} \right) \\
    \end{align*}

    \begin{align*}
        & = \sum_{1 \leq j \leq n}\sum_{i_1 \leq \cdots \leq i_j} (-1)^{j+1}\beta_u(b1_{s_{i_1}^{-1}} \cdots 1_{s_{i_j}^{-1}}) \\
        & = \sum_{1 \leq j \leq n}\sum_{i_1 \leq \cdots \leq i_j} (-1)^{j+1}\beta_u(b1_{s_{i_1}^{-1} \wedge \cdots \wedge s_{i_j}^{-1}}) \\
        & = \sum_{1 \leq j \leq n}\sum_{i_1 \leq \cdots \leq i_j} (-1)^{j+1}\beta_u(b1_{(s_{i_1} \wedge \cdots \wedge s_{i_j})^{-1}}) \\
    \end{align*}

    Let $Q = \{s_{i_1}, \ldots, s_{i_j}\}$. Then $s_{i_k} \preceq u$, for all $1 \leq k \leq j$. Therefore, all the elements of $Q$ belongs to the same $\sigma$-congruence class. Furthermore, since $S$ is $E$-unitary, by Theorem \ref{teosigmaigsim} we have $\sigma = \sim$, so $Q$ is a compatible subset of $S$.  Thus, by applying \cite[Lemma 1.4.11]{lawson1998inverse} and using a straightforward argument that confirms the meet is defined for a finite number of elements, it follows that $v_{Q} := \bigwedge_{s \in Q} s$ is defined. Since $v_Q \preceq s$, for all $s \in Q$, we have $v_Q \in S_B$. Moreover, $v_Q \preceq u$, whence $\beta_u|_{A_{v_Q^{-1}}} = \beta_{v_Q}$. Thus,
    \begin{align*}
        \beta_u(b1_{u^{-1}}) & = \sum_{1 \leq j \leq n}\sum_{i_1 \leq \cdots \leq i_j} (-1)^{j+1}\beta_{s_{i_1} \wedge \cdots \wedge s_{i_j}}(b1_{(s_{i_1} \wedge \cdots \wedge s_{i_j})^{-1}}) \\ 
        & = \sum_{1 \leq j \leq n}\sum_{i_1 \leq \cdots \leq i_j} (-1)^{j+1}b1_{s_{i_1} \wedge \cdots \wedge s_{i_j}} 
         = b\sum_{1 \leq j \leq n}\sum_{i_1 \leq \cdots \leq i_j} (-1)^{j+1}1_{s_{i_1} \wedge \cdots \wedge s_{i_j}} 
         = b1_u,
    \end{align*}
    that is, $u \in S_B$.

    Thus, by applying \cite[Lemma 1.4.11]{lawson1998inverse} and using a straightforward argument that confirms the meet is defined for a finite number of elements, it follows that \(v_Q := \bigwedge_{s \in Q} s\) is well-defined.
    
\end{proof}

From now on, we will assume that $A_s \neq 0$, for all $s \in S$.

\begin{theorem} \label{teoremagalois1}
    Suppose that $A$ is a $\beta$-Galois extension of $A^\beta$. Let $T$ be a $\beta$-complete inverse subsemigroup of $S$. Then:
    \begin{enumerate}
        \item[(i)] $\beta_T = (A_t,\beta_t)_{t \in T}$ is an action of $T$ in $A$ and $A$ is $\beta_T$-Galois over $B := A^{\beta_T}$.

        \item[(ii)] $B$ is $A^\beta$-separable.

        \item[(iii)] $B$ is $\beta$-strong.
        
        \item[(iv)] $T = S_B$.
    \end{enumerate}
\end{theorem}
\begin{proof}
    (i): The first statement is clear. Let $\{x_i,y_i\}_{i=1}^n$ be a Galois coordinate system of $A$ over $A^\beta$, that is, $\sum_{i=1}^n x_i\beta_s(y_i1_{s^{-1}}) = \sum_{e \in E(S)} 1_e\delta_{e,s}$, for all $s \in S$. In particular, as $E(S) = E(T)$, we have that $\sum_{i=1}^n x_i\beta_t(y_i1_{s^{-1}}) = \sum_{e \in E(T)} 1_e\delta_{e,t}$,
    for all $t \in T$, that is, $\{x_i,y_i\}_{i=1}^n$ is also a Galois coordinate system of $A$ over $B = A^{\beta_T}$.

    (ii): By (i), $A$ is a $\beta_T$-Galois extension of $B$. Thus, by Theorem \ref{teoequivseminvgal}(ii), $A$ is a finitely generated projective $B$-module. Therefore, there exists $p > 0$ such that $B^p = A \oplus L$, where $L$ is an appropriate $B$-module.

    Observe that $A$ and $B$ are $A^\beta$-modules via multiplication. Moreover, $A \otimes_{A^\beta} A$ is a projective $B \otimes_{A^\beta} B$-module, because
    \begin{align*}
        (B \otimes_{A^\beta} B)^{p^2} = B^p \otimes_{A^\beta} B^p = (A \oplus L) \otimes_{A^\beta} (A \oplus L) = (A \otimes_{A^\beta} A) \oplus M,
    \end{align*}
    where $M = (A \otimes_{A^\beta} L) \oplus (L \otimes_{A^\beta} A) \oplus (L \otimes_{A^\beta} L)$. Furthermore, as $A$ is $A^\beta$-separable, there are $q > 0$ and an $(A \otimes_{A^\beta} A)$-module $N$ such that $(A \otimes_{A^\beta} A)^q = A \oplus N$. Therefore,
    \begin{align*}
        (B \otimes_{A^\beta} B)^{p^2q} = (A \otimes_{A^\beta} A)^q \oplus M^q = A \oplus N \oplus M^q,
    \end{align*}
    and so $A$ is a projective $(B \otimes_{A^\beta} B)$-module.

    Furthermore, $B$ is a direct summand of $A$ as a $B$-module. Then
    \begin{align*}
        (B \otimes_{A^\beta} B)^{p^2q} = A \oplus N \oplus M^q = B \oplus \ker \tr^\sigma_{\beta_T} \oplus N \oplus M^q,
    \end{align*}
    showing that $B$ is a projective $(B \otimes_{A^\beta} B)$-module, that is, $B$ is a separable $A^\beta$-algebra.

    (iii): Since $A$ is $\beta_T$-Galois over $B$, there exists $c \in A$ such that $\tr^\sigma_{\beta_T}(c) = 1_A$ by Theorem \ref{teoequivseminvgal}(ix). Let $\alpha'$ the partial action of $G' = T/\sigma$ on $A$ constructed as in Theorem \ref{teoconstracparcsigma}. Thus, $A'_{\sigma(t)} = \sum_{s \in \sigma(t) \cap T} A_s$ and $\alpha'_{\sigma(t)} = \sum_{s \in \sigma(t) \cap T} \beta_{s} = \alpha_{\sigma(t)}|_{A'_{\sigma(t)^{-1}}}$, for all $t \in T$.

    Let $\{x_i,y_i\}_{i=1}^n$ be a $\alpha$-Galois coordinate system of $A$ on $A^\alpha$. In particular, $\{x_i,y_i\}_{i=1}^n$ is a $\alpha'$-Galois coordinate system of of $A$ on $A^{\alpha'} = A^{\beta_T} = B$ by item (i) and by Proposition \ref{propgaloissse}. 
    
    Define $x_i' = \tr^\sigma_{\beta_{T}}(cx_i)$ and $y_i' = \tr^\sigma_{\beta_{T}}(y_i)$, for $1 \leq i \leq n$. then, $x_i',y_i' \in B$, for all $1 \leq i \leq n$.

    \noindent \textbf{Claim 1:} $\displaystyle \sum_{i=1}^n x_i'y_i' = 1_A$.

    Indeed,
    \begin{align*}
        \sum_{i=1}^n x_i'y_i' & = \sum_{i=1}^n \left ( \sum_{g \in G'} \alpha_g'(cx_i1_{h^{-1}}) \right ) \left ( \sum_{h \in G'} \alpha_h'(y_i1_{h^{-1}}) \right ) 
         = \sum_{i=1}^n \sum_{g,h \in G'} \alpha_g'(cx_i \alpha_{g^{-1}}'(\alpha_{h}'(y_i1_{h^{-1}})1_{g})) \\
        & = \sum_{i=1}^n \sum_{g,h \in G'} \alpha_g'(cx_i \alpha_{g^{-1}h}'(y_i1_{h^{-1}g})1_{g^{-1}})
         = \sum_{g,h \in G'} \alpha_g' \left (c \left ( \sum_{i=1}^n x_i \alpha_{g^{-1}h}'(y_i1_{h^{-1}g}) \right ) 1_{g^{-1}} \right ) \\
        & = \sum_{g \in G'} \alpha_g' \left (c 1_A 1_{g^{-1}} \right ) 
         = \sum_{g \in G'} \alpha_g' \left (c1_{g^{-1}} \right ) 
         = \tr_{\alpha'}(c) = \tr^\sigma_{\beta_T}(c) = 1_A.
    \end{align*}

\noindent \textbf{Claim 2:} $\displaystyle \sum_{i=1}^n x_i'\beta_s(y_i'1_{s^{-1}}) = \begin{cases}
    1_s, \text{ if } s \in T, \\
    0, \text{ otherwise.}
\end{cases}$

Since $x_i',y_i' \in B = A^{\beta_T}$, we have $\sum_{i=1}^n x_i'\beta_s(y_i'1_{s^{-1}}) = \sum_{i=1}^n x_i'y_i'1_s = 1_A1_s = 1_s$, for all $s \in T$.

If $s \notin T$, then
\begin{align*}
    \sum_{i=1}^n x_i'\beta_s(y_i'1_{s^{-1}}) & = \sum_{i=1}^n \sum_{g \in G'} \alpha_g'(cx_i1_{g^{-1}}) \beta_s \left ( \sum_{h \in G'} \alpha_h'(y_i1_{h^{-1}})1_{s^{-1}} \right ) \\
    & = \sum_{i=1}^n \sum_{g \in G'} \alpha_g'(cx_i1_{g^{-1}}) \alpha_{\sigma(s)} \left ( \sum_{h \in G'} \alpha_h'(y_i1_{h^{-1}})1_{\sigma(s)^{-1}}1_{s^{-1}} \right ) \\
    & = \sum_{i=1}^n \sum_{g \in G'} \alpha_g(cx_i1_{g^{-1}}) \alpha_{\sigma(s)} \left ( \sum_{h \in G'} \alpha_h(y_i1_{h^{-1}})1_{\sigma(s)^{-1}}1_{s^{-1}} \right ) \\
    & = \sum_{i=1}^n \sum_{g \in G'} \alpha_g(cx_i1_{g^{-1}}) \left ( \sum_{h \in G'} \alpha_{\sigma(s)h}(y_i1_{h^{-1}\sigma(s)^{-1}})1_{\sigma(s)} \right )1_s
\end{align*}

\begin{align*}
    & = \sum_{i=1}^n \sum_{g \in G'} \alpha_g(cx_i1_{g^{-1}}) \left ( \sum_{h \in G'} \alpha_{\sigma(s)h}(y_i1_{h^{-1}\sigma(s)^{-1}}) \right )1_s \\
    & = \sum_{g,h \in G'} \alpha_g \left ( c \left ( \sum_{i=1}^n x_i\alpha_{g^{-1}\sigma(s)h}(y_i1_{h^{-1}\sigma(s)^{-1}g}) \right ) 1_{g^{-1}} \right ) 1_s \\
    & = \sum_{g,h \in G'} \alpha_g \left ( c \delta_{g^{-1}\sigma(s)h,1_{G'}} 1_{g^{-1}} \right ) 1_s.
\end{align*}

If $g^{-1}\sigma(s)h = 1_{G'}$, then $\sigma(s) = gh^{-1} \in G'$. So $s \in T$, a contradiction. Thus, $g^{-1}\sigma(s)h \neq 1_{G'}$, whence $\delta_{g^{-1}\sigma(s)h,1_{G'}} = 0$.

Now we shall prove that $B$ is $\beta$-strong. Indeed, suppose that $s,t \in S$ are such that $s^{-1}t \notin S_B$. In particular, $s^{-1}t \notin T$, because $T \subseteq S_B$ by definition.

Let $e \in A_s \cup A_t$ be an idempotent such that $e \neq 0$. Suppose that $\beta_s(b1_{s^{-1}})e = \beta_t(b1_{t^{-1}})e$, for all $b \in B$. Without loss of generality, consider $e \in A_s$. Hence
\begin{align*}
    b\beta_{s^{-1}}(e) = \beta_{s^{-1}}(b1_{t^{-1}s})\beta_{s^{-1}}(e), \; \text{ for all } b \in B.
\end{align*}

Define $e' = \beta_{s^{-1}}(e)$. Since $y_i' \in B$ and $\sum_{i=1}^n x_i'y_i' = 1_A$ by Claim 1, it follows that 
\begin{align*}
    e' = 1_Ae' = \sum_{i=1}^n x_i'y_i'e' = \sum_{i=1}^n x_i'\beta_{s^{-1}t}(y_i'1_{t^{-1}s})e' = 0e' = 0.
\end{align*}

As $\beta_{s^{-1}}$ is an isomorphism, we have $e = 0$, a contradiction. Therefore, there exists $b \in B$ such that $\beta_s(b1_{s^{-1}})e \neq \beta_{t}(b1_{t^{-1}})e$, that is, $B$ is $\beta$-strong.

(iv): Clearly $T \subseteq S_B$, because $B = A^{\beta_T}$. Observe that $A^{\beta_{S_B}} = A^{\beta_T} = B$. In fact, since $T \subseteq S_B$, then $A^{\beta_{S_B}} \subseteq A^{\beta_T}$. The reverse inclusion follows from the definition of $S_B$. So, from (i) it follows that $A$ is a $\beta_T$-Galois extension of $A^{\beta_{S_B}} = A^{\beta_T}$. By Theorem \ref{teoequivseminvgal}(iv),
    \begin{align*}
        \psi : A \otimes_B A & \to PA_{\beta_T}(T) \\
        x \otimes y & \mapsto (x\beta_t(y1_{t^{-1}}))_{t \in T}
    \end{align*}
    and
    \begin{align*}
        \psi' : A \otimes_B A & \to PA_{\beta_{S_B}}(S_B) \\
        x \otimes y & \mapsto (x\beta_t(y1_{t^{-1}}))_{t \in S_B}
    \end{align*}
    are isomorphisms of $A$-modules.

    Now suppose that there exists $t \in S_B \setminus T$. Consider such $t$ minimal with respect to the natural partial order, that is, if $s \prec t$, then $s \in T$ (which is always possible because $S$ is finite). Let $z = (\delta_{s,\sigma(t)} 1_s)_{s \in S_B} \in PA_{\beta_{S_B}}(S_B)$, where 
    \begin{align*}
        \delta_{s,\sigma(t)} = \begin{cases}
            1_A, \text{ if } s \in \sigma(t), \\
            0, \text{ otherwise.}
        \end{cases}
    \end{align*}

    Notice that $z \neq 0$, since $A_t \neq 0$, for all $t \in S$. As $\psi'$ is an isomorphism, there exists $0 \neq x = \sum_{i=1}^m a_i \otimes b_i \in A \otimes_B A$ such that $\psi'(x) = z$, that is, such that $\sum_{i=1}^m a_i\beta_s(b_i1_{s^{-1}}) = \delta_{s,\sigma(t)}1_s$, for all $s \in S_B$.
    
    Observe that there is not $s \in T$ with $t \preceq s$, because $T$ is an order ideal and in this case $t \in T$. Consider $U = \{u_1, \ldots, u_k\} = \{u \in T : u \prec t\}$. 
    
    \noindent \textbf{Case 1:} $U = \emptyset$.
    
    In this case, $\sigma(t) \cap T = 0$, because if there was $s \in \sigma(t) \cap T$, then $s \wedge t \in U$. Therefore, $\psi(x) = 0$, which is a contradiction to $x \neq 0$.

    \noindent \textbf{Case 2:} $U \neq \emptyset$.

    As $u_i \preceq t$, for all $1 \leq i \leq k$, we have $\beta_{u_i} \leq \beta_t$, for all $1 \leq i \leq k$, which implies $\beta_{u_i} \sim \beta_t$, for all $1 \leq i \leq k$. Thus, as $\text{Iso}_{pu}(A)$ is a $f$-complete inverse semigroup, we have that the join $f_t = \sum_{i=1}^k \beta_{u_i}$ is well-defined and $f_t \leq \beta_t$. 

    \noindent \textbf{Case 2.1:} $f_t = \beta_t$.

    In this case, we must have $t \in T$, because $T$ is $\beta$-complete, a contradiction.

    \noindent \textbf{Case 2.2:} $f_t < \beta_t$.
    
    If $f_t < \beta_t$, then $w = 1_t - 1_{\text{dom}(f_t)} \neq 0$. Hence, consider the element $wz \in PA_{\beta|_{S_B}}(S_B)$. Thus $wz = (w1_s\delta_{s,\sigma(t)})_{s \in S_B}$, and given an element $s \in S_B \cap \sigma(t)$,
    \begin{align*}
        w1_s = \begin{cases}
            w, \text{ if } t \preceq s,
            0, \text{ otherwise.}
        \end{cases}
    \end{align*}

    Indeed, if $t \preceq s$, then $w1_s = (1_t - 1_{\text{dom}(f_t)})1_s = 1_t1_S - 1_{\text{dom}(f_t)}1_s = 1_t - 1_{\text{dom}(f_t)} = w$. Now, if $t \preceq s$ does not hold, then $t \wedge s \preceq t$, so $s \wedge t = u_j$, for some $1 \leq j \leq k$, since $t$ is minimal with respect to the natural partial order. Thus, $1_{u_i}1_s = (1_{u_i}1_t)1_s = 1_{u_i}(1_t1_s) = 1_{u_i}1_{u_j}$, for all $1 \leq i \leq k$. Then, $1_{\text{dom}(f_t)}1_s = 1_{u_j}$, because
    \begin{align*}
        1_{\text{dom}(f_t)}1_s & = \left ( \sum_{\ell=1}^k \sum_{i_1 \leq \cdots \leq i_\ell} (-1)^{k+1}1_{u_{i_1}} \cdots 1_{u_{i_\ell}} \right)1_s 
         = \sum_{\ell=1}^k \sum_{i_1 \leq \cdots \leq i_\ell} (-1)^{k+1}1_{u_{i_1}} \cdots 1_{u_{i_\ell}}1_s \\
        & = \sum_{\ell=1}^k \sum_{i_1 \leq \cdots \leq i_\ell} (-1)^{k+1}1_{u_{i_1}} \cdots 1_{u_{i_\ell}}1_{u_j} 
         = \left ( \sum_{\ell=1}^k \sum_{i_1 \leq \cdots \leq i_\ell} (-1)^{k+1}1_{u_{i_1}} \cdots 1_{u_{i_\ell}} \right )1_{u_j} \\
         & = 1_{\text{dom}(f_t)}1_{u_j} = 1_{u_j}.
    \end{align*}

    Hence, $w1_s = 1_t1_s - 1_{\text{dom}(f_t)}1_s = 1_{u_j} - 1_{u_j} = 0$. 

    Therefore, there exists $0 \neq x' \in A \otimes_B A$ such that $\psi'(x') = wz$. Assume that $x' = \sum_{i=1}^{m'} a_i' \otimes b_i'$. Thus,
    \begin{align*}
        \sum_{i=1}^{m'} a_i'\beta_s(b_i') = \begin{cases}
            w, \text{ if } t \preceq s, \\
            0, \text{ otherwise.}
        \end{cases}
    \end{align*}

    Observe now that $\psi(x') = 0$, because there is not $s \in T$ with $t \preceq s$. But $\psi$ is an isomorphism and $x' \neq 0$, a contradiction.
\end{proof}

\begin{lemma} \label{lemagammatensor}
    Suppose that $A$ is a $\beta$-Galois extension of $A^\beta$ and let $R$ be a commutative $A^\beta$-algebra. Define $\gamma = (R \otimes A_s, \gamma_s)_{s \in S}$ as an action of $S$ on $R \otimes_{A^{\beta}} A$ induced by $\beta$ via $\gamma_s(r \otimes a1_{s^{-1}}) = r \otimes \beta_s(a1_{s^{-1}})$, for $r \in R$, $a \in A$ and $s \in S$. Then $R \otimes_{A^\beta} A$ is a $\gamma$-Galois extension of $R$.
\end{lemma}

\begin{proof}
    We have $A \simeq A^\beta \oplus \ker(\tr_\beta^\sigma)$. Then $R \otimes_{A^\beta} A \simeq (R \otimes_{A^\beta} A^\beta) \oplus (R \otimes_{A^\beta} \ker(\tr_\beta^\sigma))$. We will identify $R$ with its homomorphic image $R \otimes_{A^\beta} A^\beta$.

    If $\{x_i,y_i\}_{i=1}^n \subseteq A$ are a $\beta$-Galois coordinate system of $A$ over $A^\beta$, then clearly teh set $\{1_R \otimes x_i, 1_R \otimes y_i\}_{i=1}^n$ satisfies
    \begin{align*}
        \sum_{i=1}^n (1_R \otimes x_i)\gamma_s(1_R \otimes y_i1_{s^{-1}}) = \sum_{e \in E(S)} \delta_{e,s}(1_R \otimes 1_e),
    \end{align*}
    for all $s \in S$. Therefore, it remains to prove that $(R \otimes_{A^\beta} A)^\gamma = R$.

    Consider $u = \sum_{i=1}^m r_i \otimes a_i \in (R \otimes A)^\gamma$ and $c \in R$ such that $\tr_\beta^\sigma(c) = 1_A$. Then
    \begin{align*}
        u & = u(\text{Id}_R \otimes \tr_\beta^\sigma)(1_R \otimes c) 
         = \left ( \sum_{i=1}^m r_i \otimes a_i \right ) \left ( 1_R \otimes \tr_\beta^\sigma(c) \right ) 
         = \sum_{i=1}^m r_i \otimes a_i\tr_\beta^\sigma(c) \\
        & = \sum_{i=1}^m r_i \otimes a_i\sum_{g \in G} \alpha_g(c1_{g^{-1}}) 
         = \sum_{g \in G} \sum_{i=1}^m (r_i \otimes a_i)\alpha_g(c1_{g^{-1}}) 
         = \sum_{g \in G} \sum_{i=1}^m (r_i \otimes a_i)1_g\alpha_g(c1_{g^{-1}}) \\
        & = \sum_{g \in G} \sum_{i=1}^m (r_i \otimes a_i1_g)\alpha_g(c1_{g^{-1}}) 
         = \sum_{g \in G} \sum_{i=1}^m (r_i \otimes \alpha_g(a_i1_{g^{-1}}))\alpha_g(c1_{g^{-1}}) \\
        & = \sum_{i=1}^m r_i \otimes \sum_{g \in G} \alpha_g(a_i1_{g^{-1}})\alpha_g(c1_{g^{-1}}) 
         = \sum_{i=1}^m r_i \otimes \sum_{g \in G} \alpha_g(a_ic1_{g^{-1}}) 
         = \sum_{i=1}^m r_i \otimes \tr_\beta^\sigma(a_ic),
    \end{align*}
    that is, $u \in R \otimes_{A^\beta} A^\beta = R$.

    The reverse inclusion is immediate.
\end{proof}

Define $E = \{ f : S \to A : f \text{ is a function, } f(s) \in A_s, \text{ for all } s \in S, \text{ and if } s \preceq t, \text{ then } $ $ f(s) = f(t)1_s \}$. Note that $E$ is a $A^\beta$-algebra with pointwise operations.

\begin{lemma} \label{lemaacaofuncoes}
    Suppose that $A$ is $\beta$-Galois over $A^\beta$. Then there exists an action $\beta'$ of $S$ on $E$ and $E$ is $\beta'$-Galois over $A$.
\end{lemma}
\begin{proof}
    By Lemma \ref{lemagammatensor}, $A \otimes_{A^\beta} A$ is $\gamma$-Galois on $A$, where $\gamma_s(a \otimes b1_{s^{-1}}) = a \otimes \beta_s(b1_{s^{-1}})$, for $a, b \in A$. But $E \simeq PA_\beta(S)$ via
    \begin{align*}
        \phi : E & \to PA_\beta(S) \\
        f & \mapsto (f(s))_{s \in S}
    \end{align*}
    with inverse given by
    \begin{align*}
        \phi^{-1} : PA_\beta(S) & \to E \\
        (a_s)_{s \in S} & \mapsto f : S \to R, \qquad f(s) = a_s, \; \forall s \in S.
    \end{align*}

    By item (iv) of Theorem \ref{teoequivseminvgal},
    \begin{align*}
        \psi : A \otimes_{A^\beta} A & \to PA_\beta(S) \\
        a \otimes b & \mapsto (a\beta_s(b1_{s^{-1}}))_{s \in S}
    \end{align*}
    is an isomorphism of $A^\beta$-algebras. Thus, the map $\eta = \phi^{-1} \circ \psi$ given by
    \begin{align*}
        \eta : A \otimes_{A^\beta} A & \to E \\
            a \otimes b & \mapsto \eta(a \otimes b) : S \to R, \qquad \eta(a \otimes b)(s) = a\beta_s(b1_{s^{-1}}),
    \end{align*}
    is an isomorphism of $A^\beta$-algebras.

    This isomorphism induces an action $\beta' = (E_s,\beta_s')_{s \in S}$ of $S$ in $E = \eta(A \otimes_{A^\beta} A)$ given by
    \begin{align*}
        \beta_s'(\eta(a \otimes b)\eta(1_A \otimes 1_{s^{-1}}))(t) & = \beta_s'(\eta(a \otimes b1_{s^{-1}}))(t) 
         = \eta(\gamma_s(a \otimes b1_{s^{-1}}))(t) \\
        & = \eta(a \otimes \beta_s(b1_{s^{-1}}))(t) 
         = a\beta_t(\beta_s(b1_{s^{-1}})1_{t^{-1}}),
    \end{align*}
    where the ideals are given by $E_s = \eta(A \otimes_{A^\beta} A_s)$. Now, as $A \otimes_{A^\beta} A$ is $\beta$-Galois on $A$, it is easy to see that $E$ is $\beta'$-Galois on $A$.
\end{proof}

\begin{lemma}
    Assume $A$ a $\beta$-Galois extension of $A^\beta$ and $T$ an inverse subsemigroup of $S$. Then $f \in E^{\beta'|_{T}}$ if and only if $f(st) = f(s)1_{st}$, for all $t \in T$ and $s \in S$.
\end{lemma}
\begin{proof}
Consider $f \in E^{\beta'|_{T}}$. Since $\eta$, as defined in Lemma \ref{lemaacaofuncoes}, is an isomorphism, there exists $\sum_{i=1}^n a_i \otimes b_i \in A \otimes_{A^\beta} A$ such that $f = \eta \left ( \sum_{i=1}^n a_i \otimes b_i \right )$. Thus,
\begin{align*}
    \eta \left ( \sum_{i=1}^n a_i \otimes b_i \right ) \in E^{\beta'|_{T}} & \Leftrightarrow \beta_t'\left ( \eta \left ( \sum_{i=1}^n a_i \otimes b_i1_{t^{-1}} \right ) \right ) = \eta \left ( \sum_{i=1}^n a_i \otimes b_i1_t \right ), \; \forall t \in T \\
    & \Leftrightarrow \eta \left ( \sum_{i=1}^n a_i \otimes \beta_t(b_i1_{t^{-1}}) \right ) = \eta \left ( \sum_{i=1}^n a_i \otimes b_i1_t \right ), \; \forall t \in T \\
    & \Leftrightarrow \eta \left ( \sum_{i=1}^n a_i \otimes \beta_t(b_i1_{t^{-1}}) \right )(u) = \eta \left ( \sum_{i=1}^n a_i \otimes b_i1_t \right )(u), \; \forall t \in T, u \in S \\
    & \Leftrightarrow \sum_{i=1}^n a_i\beta_u(\beta_t(b_i1_{t^{-1}})1_{u^{-1}}) = \sum_{i=1}^n a_i\beta_u(b_i1_t1_{u^{-1}}), \; \forall t \in T, u \in S \\
    & \Leftrightarrow \sum_{i=1}^n a_i\beta_{ut}(b_i1_{t^{-1}u^{-1}})1_u = \sum_{i=1}^n a_i\beta_u(b_i1_{u^{-1}})1_{ut}, \; \forall t \in T, u \in S \\
    & \Leftrightarrow \eta\left ( \sum_{i=1}^n a_i \otimes b_i \right )(ut)1_u = \eta\left( \sum_{i=1}^n a_i \otimes b_i \right )(u)1_{ut}, \; \forall t \in T, u \in S \\
    & \Leftrightarrow f(ut)1_u = f(u)1_{ut}, \; \forall t \in T, u \in S \\
    & \Leftrightarrow f(ut)1_{ut}1_u = f(u)1_{ut}, \; \forall t \in T, u \in S \\
    & \Leftrightarrow f(ut)1_{ut} = f(u)1_{ut}, \; \forall t \in T, u \in S \\
    & \Leftrightarrow f(ut) = f(u)1_{ut}, \; \forall t \in T, u \in S.
\end{align*}
\end{proof}

For the next theorem, we recall the reader the definition of restricted product, denoted $\cdot$, in an inverse semigroup. The \emph{restricted product} $s \cdot t$ is defined in $S$ if and only if $s^{-1}s = tt^{-1}$ and, in this case, $s \cdot t = st$.

 Let $T$ be a full inverse subsemigroup of $S$. Define the relation $\equiv_T$ in $S$ by \begin{align}\label{quotient} s \equiv_T u \Leftrightarrow \exists u^{-1} \cdot s \text{ and } u^{-1} \cdot s \in T.\end{align} Clearly this relation is reflexive (since $T$ is full) and symmetric. To show the transitivity, take $s \equiv_T u$ and $u \equiv_T v$. Then $u^{-1} \cdot s, v^{-1} \cdot u \in T$. Since $T$ is an inverse subsemigroup, it follows that $(v^{-1}\cdot u)(u^{-1} \cdot s) \in T$. Now, notice that $(v^{-1} \cdot u)(u^{-1} \cdot s) = (v^{-1}u)(u^{-1}s) = v^{-1}(uu^{-1})s = v^{-1}(vv^{-1})s = (v^{-1}vv^{-1})s = v^{-1}s$, and $vv^{-1} = uu^{-1} = ss^{-1}$, hence $\exists v^{-1} \cdot s$ and $v^{-1} \cdot s \in T$.

\begin{theorem} \label{teoremagalois2}
    Suppose that $A$ is a $\beta$-Galois extension of $A^\beta$. Let $B$ be an $A^\beta$-separable and $\beta$-strong subalgebra of $A$. Then $A^{\beta|_{S_B}} = B$.
\end{theorem}
\begin{proof}
    Let $T = S_B$. We have $B \subseteq A^{\beta|_{T}}$. Hence we shall prove that  $A^{\beta|_{T}} \subseteq B$.

    \noindent \textbf{Claim:} $E^{\beta'|_{T}} \subseteq \eta(A \otimes_{A^\beta} B)$.

    Let $\{s_i\}_{i=1}^n$ a system of representative of the congruence classes defined by $\equiv_T$. Consider the homomorphism of $A$-algebras $f_i : E \to A$ defined by $f_i(v) = v(s_i)$.

    We will proceed to prove that $f_1, \ldots, f_n$ are strongly distinct. Indeed, if $i \neq j$, then the restrictions $\beta_{s_i}|_B$ and $\beta_{s_j}|_B$ can not coincide. In fact, if $\beta_{s_i}(b1_{s_i^{-1}}) = \beta_{s_j}(b1_{s_j}^{-1})$, for all $b \in B$, then $\beta_{s_j}^{-1}(\beta_{s_i}(b1_{s_i^{-1}})1_{s_j}) = b1_{s_j^{-1}}$. Hence,
    \begin{align*}
        & \beta_{s_j^{-1}s_i}(b1_{s_i^{-1}s_j})1_{s_j^{-1}} = t1_{s_j^{-1}} 
        \Rightarrow  \beta_{s_j^{-1}s_i}(b1_{s_i^{-1}s_j})1_{s_j^{-1}s_i}1_{s_j^{-1}} = t1_{s_j^{-1}}1_{s_j^{-1}s_i} \\
        \Rightarrow & \beta_{s_j^{-1}s_i}(b1_{s_i^{-1}s_j})1_{s_j^{-1}s_i} = t1_{s_j^{-1}s_i} 
        \Rightarrow  \beta_{s_j^{-1}s_i}(b1_{s_i^{-1}s_j}) = t1_{s_j^{-1}s_i},
    \end{align*}
    for all $b \in B$, that is, $s_j^{-1}s_i \in S_B = T$, whence $s_i \equiv_T s_j$, a contradiction.

    Now, since $B$ is $\beta$-strong, for each nonzero idempotent $e \in A_{s_i} \cup A_{s_j}$ there exists $b \in B$ such that $\beta_{s_i}(b1_{s_i^{-1}})e \neq \beta_{s_j}(b1_{s_j^{-1}})e$. Thus,
    \begin{align*}
        f_i(\eta(1_A \otimes b))e = \eta(1_A \otimes b)(s_i)e = \beta_{s_i}(b1_{s_i^{-1}})e \neq \beta_{s_j}(b1_{s_j^{-1}})e = f_j(\eta(1_A \otimes b))e.
    \end{align*}

    Now we can prove the claim. Since $B$ is $A^\beta$-separable, then $A \otimes_{A^\beta} B$ is $A$-separable by \cite[Proposition III 2.1]{knus2006theorie} and so $\eta(A \otimes_{A^\beta} B)$ is $A$-separable. From \cite[Lemma 2.4]{paques2018galois}, we obtain pairwise orthogonal idempotents  $w_1, \ldots, w_n \in \eta(A \otimes_{A^\beta} B)$ with $f_i(x)w_i = xw_i$, for all $x \in \eta(A \otimes_{A^\beta} B)$ and $w_j(s_i) = f_i(w_j) = \delta_{i,j}$, for all $1 \leq i,j \leq n$.

    We shall verify that $w_1, \ldots, w_n$ generates $E^{\beta'|_T}$ on $A$. To this, let us check firstly that $w_i \in E^{\beta'|_T}$, for all $1 \leq i \leq n$. Since $w_i \in \eta(A \otimes_{A^\beta} B)$, there exist $a_{ij} \in A$, $b_{ij} \in B$ such that $w_i = \eta \left ( \sum_{j} a_{ij} \otimes b_{ij} \right )$, for $1 \leq i \leq n$. We will prove that $w_i(st) = w_i(s)1_{st}$, for all $s \in S$ and $t \in T$. Indeed,
    \begin{align*}
        w_i(st) & = \eta \left ( \sum_{j} a_{ij} \otimes b_{ij} \right )(st) = \sum_j a_{ij}\beta_{st}(b_{ij}1_{t^{-1}s^{-1}}) \\
        & = \sum_j a_{ij}\beta_{s}(\beta_t(b_{ij}1_{t^{-1}})1_{s^{-1}}) = \sum_j a_{ij}\beta_s(b_{ij}1_t1_{s^{-1}}) \\
        & = \sum_j a_{ij}\beta_{s}(b_{ij}1_{s^{-1}})1_{st} = \eta \left ( \sum_{j} a_{ij} \otimes b_{ij} \right )(s)1_{st} = w_i(s)1_{st}.
    \end{align*}

    Therefore $w_i \in E^{\beta'|_T}$, for all $1 \leq i \leq n$. Consider now $f \in E^{\beta'|_T}$. then there are $a_j,b_j \in A$ such that $f = \eta \left ( \sum_j a_j \otimes b_j \right )$. Hence,
    \begin{align*}
        f(s_i) & = \eta \left ( \sum_j a_j \otimes b_j \right )(s_i) = \sum_j a_j\beta_{s_i}(b_j1_{s_i^{-1}}) \\
        & = \sum_{j,k} s_j\beta_{s_k}(b_j1_{s_k^{-1}})\delta_{k,i} = \sum_{j,k} s_j\beta_{s_k}(b_j1_{s_k^{-1}})w_k(s_i).
    \end{align*}

    As $f \in E^{\beta'|_T}$, it follows that $f(s_i)1_{s_it} = f(s_it)$, for all $t \in T$. Thus,
    \begin{align*}
        f(s_it) = f(s_i)1_{s_it} = \sum_{j,k} s_j\beta_{s_k}(b_j1_{s_k^{-1}})w_k(s_i)1_{s_it} = \sum_{j,k} s_j\beta_{s_k}(b_j1_{s_k^{-1}})w_k(s_it),
    \end{align*}
    and so $f = \sum_{j,k} s_j\beta_{s_k}(b_j1_{s_k^{-1}})w_k$, that is, $w_1, \ldots, w_n$ generates $E^{\beta'|_T}$ on $A$. Therefore $E^{\beta'|_T} \subseteq \eta(A \otimes_{A^\beta} B)$, which concludes the claim.

    It remains to prove that  $A^{\beta|_T} \subseteq B$. In fact, applying $\eta^{-1}$ in the Claim, we have $\eta^{-1}(E^{\beta'|_T}) \subseteq A \otimes_{A^\beta} B$. Note that $E^{\beta'|_T} = \eta((A \otimes_{A^\beta} A)^{\gamma|_T})$, because $\beta'_s \otimes \eta = \eta \otimes \gamma_s$, for all $s \in S$. Thus, $(A \otimes_{A^\beta} A)^{\gamma|_{T}} = \eta^{-1}(E^{\beta'|_{T}}) \subseteq A \otimes_{A^\beta} B$. Hence,
    \begin{align*}
        A \otimes_{A^\beta} A^{\beta|_{T}} \subseteq (A \otimes_{A^\beta} A)^{\gamma|_T} \subseteq A \otimes_{A^\beta} B.
    \end{align*}

    Applying $\tr_\beta^\sigma \otimes_{A^\beta} \text{Id}_A$ in this inclusion, we have that
    \begin{align*}
        A^{\beta|_{T}} = A^\beta \otimes_{A^\beta} A^{\beta|_{T}} = \tr_\beta^\sigma(A) \otimes_{A^\beta} A^{\beta|_{T}} \subseteq \tr_\beta^\sigma(A) \otimes_{A^\beta} B = A^\beta \otimes_{A^\beta} B = B,
    \end{align*}
    since $\tr_\beta^\sigma$ is surjective in $A^\beta$. This concludes the proof.
\end{proof}

Now we are able to enunciate our first main result, which establishes a one-to-one correspondence between the $A^\beta$-separable and $\beta$-strong subalgebras of $A$ and the $\beta$-complete inverse subsemigroups of $S$.

\begin{theorem}\textbf{Galois Correspondence for $E$-unitary inverse semigroups.} \label{teocorrgalparceunit}
   Let $S$ be a finite inverse semigroup $E$-unitary that acts injectively on a commutative ring $A$ via a unital action $\beta = (A_s,\beta_s)_{s \in S}$. Suppose that $A$ is a $\beta$-Galois extension of $A^\beta$ such that $A_s \neq 0$, for all $s \in S$. Then there exists a one-to-one correspondence between the $A^\beta$-separable and $\beta$-strong subalgebras $B$ of $A$ and the $\beta$-complete inverse subsemigroups $T$ of $S$ given by $B \mapsto S_B$ and $T \mapsto A^{\beta|_{T}}$.
\end{theorem}
\begin{proof}
It follows from Theorem \ref{teoremagalois1} and Theorem \ref{teoremagalois2}.

\end{proof}

\section{General case}

We now extend the Galois correspondence to general inverse semigroups (not necessarily $E$-unitary) acting on commutative rings (not necessarily injectively).  Throughout this section, let $S$ be a finite inverse semigroup without zero, and let $\beta = (A_s,\beta_s)_{s \in S}$ denote a unital action of $S$ on a commutative ring $A$. Observe that $\beta$ can be regarded as an inverse semigroup homomorphism $\beta : S \to \text{Iso}_{pu}(A)$ defined by $\beta(s) := \beta_s$. We will consistently use this notation throughout the section.

\begin{prop}
    Assume that $A$ is $\beta$-Galois over $A^\beta$ and that $A_s \neq 0$, for all $s \in S$. Then $\beta(S) = \{ \beta_s : s \in S \}$ is an $E$-unitary inverse semigroup.
\end{prop}
\begin{proof}
    Suppose that $e \in E(S)$ and $\beta_e = \beta_s|_{A_e}$, for some $s \in S$. If $s \in E(S)$, there is nothing to prove. If $s \notin E(S)$, then
    \begin{align*}
        1_e & = \sum_{i=1}^n x_i\beta_e(y_i1_e) = \sum_{i=1}^n x_i\beta_s(y_i1_e) \\
        & = \sum_{i=1}^n x_i\beta_s(y_i1_{s^{-1}})\beta_s(1_e)  = \sum_{i=1}^n x_i\beta_s(y_i1_{s^{-1}})1_e = 01_e = 0.
    \end{align*}

    But by hypothesis $1_e \neq 0$, from where we obtain a contradiction. Thus $s \in E(S)$ and $\beta_s \in E(\beta(S))$.
\end{proof}

This result is essential for extending Theorem \ref{teocorrgalparceunit} to the general case. Additionally, we will rely on the following lemma. While it is important, its proof is straightforward and will therefore be omitted.

\begin{lemma}
    The inverse semigroup $\beta(S)$ acts on $A$ via $\beta(s) \mapsto \beta_s$. If $T$ is an inverse subsemigroup of $\beta(S)$, then $T$ is $\beta$-complete if and only if $T$ is $f$-complete.
\end{lemma}

We have chosen to preserve the terminology ``$f$-complete". However, it is relevant to note that once $T$ is finite, the concepts of complete and $f$-complete coincide. 

Finally, we will identify the class of inverse subsemigroups that will appear in the Galois correspondence.

\begin{defi}
    A full inverse subsemigroup $T$ of $S$ is called \emph{$\beta$-maximal} if $\beta(T)$ is $f$-complete, and  for all $t \in T$, we have $\beta^{-1}(t) = \{s \in S : \beta_s = \beta_t\} \subseteq T$.
\end{defi}

\begin{theorem}\textbf{Galois Correspondence.}
    Let $S$ be a finite inverse semigroup that acts on a commutative ring $A$ via a unital $\beta = (A_s,\beta_s)_{s \in S}$. Assume that $A$ is $\beta$-Galois over $A^\beta$ and that $A_s \neq 0$, for all $s \in S$. Then there is a one-to-one correspondence between the $\beta$-maximal inverse subsemigroups $T$ of $S$ and the separable, $\beta$-strong $A^\beta$-subalgebras $B$ of $A$ given by $B \mapsto \beta^{-1}(\beta(S)_B)$ with inverse $T \mapsto A^{\beta|_{T}}$.
\end{theorem}
\begin{proof}
    Notice first that $A$ is $\beta$-Galois over $A^\beta$ regarding the action $\beta$ of $S$ on $A$ if and only if $A$ is $\beta'$-Galois over $A^{\beta'}$ regarding the action $\beta'$ of $\beta(S)$ on $A$. It is easy to see that $A^\beta = A^{\beta'}$. 

    Now, $\beta(S)$ is an $E$-unitary finite inverse semigroup without zero that acts injectively on $A$ via $\beta'$. Thus we are on the hypotheses of Theorem \ref{teocorrgalparceunit}. Hence there is a one-to-one correspondence between the $f$-complete inverse subsemigroups of $\beta(S)$ and the separable, $\beta$-strong $A^\beta$-subalgebras of $A$. The verification that a subalgebra is $\beta$-strong if and only if it is $\beta'$-strong is immediate. 

    Finally, it only remains to us to observe that there is a one-to-one correspondence between the $\beta$-maximal inverse subsemigroups of $S$ and the $f$-complete inverse subsemigroups of $\beta(S)$. Then the result follows.
\end{proof}

\section{Inverse semigroups with zero}

This section is dedicated to discuss how the theory behaves in the case of inverse semigroups with zero. We will maintain the notations from the previous section, except that now $S$ will be assumed to be an inverse semigroup with zero.

 Given a subset $T \subseteq S$, we set:
\begin{align*}
    T^* & = \begin{cases}
        T \setminus \{0\}, \text{ if } 0 \in T, \\
        T, \text{ if } 0 \notin T.
    \end{cases}
\end{align*}

A pair of elements $s,t \in S$ is said \emph{strongly compatible} if $s = t = 0$ or if $s,t \neq 0$ and $s^{-1}t, st^{-1} \in E(S)^*$. In this case, we write $s \approx t$. Clearly $\approx$ is reflexive and symmetric, but not necessarily transitive.

We say that an inverse semigroup with zero is \emph{0-$E$-unitary} if $e \preceq s$, for some $e \in E(S)^*$, implies $s \in E(S)$.

\begin{lemma} \cite[Lemma 9.1.1]{lawson1998inverse}\label{lemma911}
    Let $S$ be an inverse semigroup with zero.

    \begin{enumerate}
        \item[(i)] If $s, t \in S^*$ and $s \approx t$, then $s$ and $t$ have a non-zero lower bound.

        \item[(ii)] Assume that $S$ is 0-$E$-unitary. Then a pair of elements $s,t \in S^*$ has a non-zero lower bound if and only if $s \approx t$.
    \end{enumerate}
\end{lemma}

\begin{prop} \label{propcaracapprox}
    Let $S$ be an inverse semigroup with zero.

    \begin{enumerate}
        \item[(i)] If $s, t \in S^*$ and $s \approx t$, then $s \wedge t$ is defined and $s \wedge t \in S^*$. Furthermore, $(s \wedge t)(s \wedge t)^{-1} = ss^{-1}tt^{-1}$ and $(s \wedge t)^{-1}(s \wedge t) = s^{-1}st^{-1}t$.

        \item[(ii)] Assume that $S$ is 0-$E$-unitary. Then a pair of elements $s,t \in S^*$ is such that $s \wedge t \in S^*$ if and only if $s \approx t$.
    \end{enumerate}
\end{prop}
\begin{proof}
    By the proof of \cite[Lemma 9.1.1]{lawson1998inverse}, the non-zero lower bound of $s$ and $t$ is $z = ss^{-1}t$.

\noindent \emph{Claim.} $z = s \wedge t$.

Indeed, if $w \preceq s,t$, then $w = ww^{-1}w \preceq ss^{-1}t = z$.

Moreover, we have that $zz^{-1} = ss^{-1}tt^{-1}$. Since $st^{-1}t$ is also a meet of $s$ and $t$, it follows that $z = st^{-1}t$, so $z^{-1}z = s^{-1}st^{-1}t$. Clearly $zz^{-1}$ and $z^{-1}z$ are non-zero, otherwise $z = 0$. Now the result follows by Lemma \ref{lemma911}.
    
\end{proof}

We say that a poset is \emph{directed} if any two elements have a lower bound. This concept characterizes when the relation $\approx$ is transitive, as we can see below.

\begin{prop} \cite[Theorem 9.1.2]{lawson1998inverse} \label{prop73}
    Let $S$ be an inverse semigroup with zero. Then $\approx$ is transitive if and only if $S$ is 0-$E$-unitary and for all $s \in S^*$ the set $[s]^*$ is directed.
\end{prop}

An inverse semigroup with zero $S$ is said \emph{categorical at zero} if $stu = 0$ implies that $st = 0$ or $tu = 0$, for $s,t,u \in S$.

\begin{lemma} \cite[Lemma 9.1.3]{lawson1998inverse} \label{lema74}
    Let $S$ be an inverse semigroup with zero categorical at zero. Then the set $[s]^*$ is directed for all $s \in S^*$.
\end{lemma}

Let $S$ be an inverse semigroup with zero. We define the relation $\tau$ in $S$ by $s \tau t$ if $s = t = 0$ or $s,t \in S^*$ and there is an element $u \in S^*$ such that $u \preceq s,t$. The relation $\tau$ works as the minimum group congruence $\sigma$ for the case of inverse semigroups with zero.

\begin{prop} \label{tauapprox}$\tau = \; \approx$ in 0-$E$-unitary inverse semigroups categorical at zero. \end{prop}

\begin{proof}
    It is an immediate consequence of Propositions \ref{propcaracapprox}, \ref{prop73} and Lemma \ref{lema74}.
\end{proof}

A congruence is said \emph{0-restricted} if the class of 0 only contains 0. An inverse semigroup with zero $S$ is said \emph{primitive} if $\preceq$ is equality in $S^*$. A congruence is said \emph{primitive} if the quotient by this congruence is a primitive inverse semigroup.

\begin{prop} \cite[Proposition 9.1.4]{lawson1998inverse}
    Let $S$ be an inverse semigroup with zero categorical at zero. Then $\tau$ is the minimum primitive $0$-restricted congruence in $S$.
\end{prop}

Next, we will show that the partial actions of primitive inverse semigroups are intrinsically related to orthogonal partial actions of groupoids. Before that, we shall recall some definitions regarding groupoids.

Consider a set $\G$ equipped with a partially defined operation denoted by concatenation. We write $\G_2 = \{ (g,h) \in \G \times \G : \text{the product } gh \text{ is defined}\}$. We say that $e \in \G$ is an \emph{identity} if $(e,e) \in \G_2$, $e^2 = e$ and if $(g,e) \in \G_2$ (resp. $(e,g) \in \G_2$) then $ge = g$ (resp. $eg = g$). We denote by $\G_0$ the set of identities of $\G$. We say that $\G$ is a \emph{groupoid} if the following conditions hold:
\begin{itemize}
\item [(i)] for all $g, h, l \in \G$, $(g,hl) \in \G_2$ if and only if $(gh,l) \in \G_2$ and in this case $(gh)l = g(hl)$;
    \item[(ii)] for all $g, h, l \in \G$, $(g,hl) \in \G_2$  if and only if $(g,h) \in \G_2$ and $(h,l) \in \G_2$;
        \item[(iii)] for each $g \in \G$, there exist (unique) identities $d(g), r(g) \in \G_0$ such that $(g,d(g)) \in \G_2$ and $(r(g),g) \in \G_2$;
            \item[(iv)] for each $g \in \G$ there exists $g^{-1} \in \G$ such that $(g,g^{-1}),(g^{-1},g) \in \G_2$, $d(g) = g^{-1}g$ and $r(g) = gg^{-1}$. \end{itemize}

Every groupoid gives rise to a primitive inverse semigroup and vice versa. The structures are defined in the following way. Let $\G$ be a groupoid. Defining $\G^0 = \G \cup \{ 0 \}$ with product $g0 = 0 = 0g$, for all $g \in \G^0$, and $gh = 0$ if $\nexists gh$, we have that $\G^0$ is a primitive inverse semigroup. In fact, by \cite[Theorem 3.3.4]{lawson1998inverse}, every primitive inverse semigroup is obtained in this way. Conversely, let $S$ be  a primitive inverse semigroup. Place in the set $S^*$ a groupoid structure  given by $\exists s \cdot t$ if and only if $st \neq 0$, and in this case $s \cdot t = st$.

Our next result requires the definitions of partial actions for inverse semigroups and groupoids. We adopt the definition of partial inverse semigroup action as stated in \cite[Definition 2.1]{beuter2019simplicity} with a slight weakening of the original assumptions. Here, we assume that $A_s \triangleleft A_{ss^{-1}}$ instead of $A_s \triangleleft A$, for each $s \in S$.

\begin{defi}\label{def-par-semi}
We define $\alpha = (A_s,\alpha_s)_{s \in S}$ as a partial action of an inverse semigroup $S$ on a ring $A$ if:

1. $A_{ss^{-1}} \triangleleft A$, $A_s \triangleleft A_{ss^{-1}}$, and $\alpha_s : A_{s^{-1}} \to A_s$ is an isomorphism  of rings, for all $s \in S$;

2. The following conditions hold:\begin{itemize}
    \item[(PIS1)] $A = \sum_{e \in E(S)} A_e$;
    
    \item[(PIS2)] $\alpha_{t}^{-1}(A_{s^{-1}} \cap A_t) \subseteq A_{(st)^{-1}}$, for all $s,t \in S$;
    
    \item[(PIS3)] $\alpha_s \circ \alpha_t(x) = \alpha_{st}(x)$, for all $s,t \in S, x \in \alpha_t^{-1}(A_{s^{-1}}\cap A_t)$.
\end{itemize}
\end{defi}  

Next we introduce the notion of a partial inverse semigroup with zero action.

\begin{defi}\label{def-par-semi-zero}Let $S$ be an inverse semigroup with zero and $\alpha = (A_s,\alpha_s)_{s \in S}$ a partial action of $S$ on $A$. We say that $\alpha$ is a \emph{partial inverse semigroup with zero action} if:
\begin{enumerate}
    \item[(PIS0)] $A_0 = 0$ and $\alpha_0 = 0$.
\end{enumerate}\end{defi}

A \emph{unital partial action} of $S$ on $A$ is a partial action $\alpha$ of $S$ on $A$ such that every ideal $A_s$, for $s \in S$, is generated by a central idempotent $1_s$ of $A$.

\begin{defi}\label{par-groupoid}
Let $\G$ be a groupoid and $A$ be a ring. We define $\alpha = (A_g,\alpha_g)_{g \in \G}$ as a \emph{partial groupoid action} of $\G$ on $A$ if:

1.$A_{r(g)} \triangleleft A$, $A_g \triangleleft A_{r(g)}$, and $\alpha_g : A_{g^{-1}} \to A_g$ is an isomorphism of rings, for all $g \in \G$;

2. The following conditions hold:\begin{itemize}
    \item[(PGr1)] $A = \sum_{e \in \G_0} A_e$;
    \item[(PGr2)] $\alpha_{h}^{-1}(A_{g^{-1}} \cap A_h) \subseteq A_{(gh)^{-1}}$, for all $(g,h) \in \G_2$;
    \item[(PGr3)] $\alpha_g \circ \alpha_h(a) = \alpha_{gh}(a)$, for all $(g,h) \in \G_2, a \in \alpha_h^{-1}(A_{g^{-1}}\cap A_h)$.
\end{itemize}\end{defi}

\begin{defi} We say that a partial action $\alpha$ is an \emph{orthogonal} partial groupoid action if holds:
\begin{enumerate}
    \item[(PGr0)] $A = \displaystyle\bigoplus_{e \in \G_0} A_e$.
\end{enumerate}
\end{defi}

A \emph{unital partial action} of $\G$ on $A$ is a partial action $\alpha$ such that every $A_g$, for $g \in \G$, is generated by a central idempotent $1_g$ of $A$.

\begin{obs} We highlight two facts concerning the definition of partial action of a groupoid. \begin{itemize}
\item[1.] The concept of a partial groupoid action was first defined by Bagio and Paques in \cite{bagio2012partial}, although the authors did not require the condition $A = \sum_{e \in \G_0} A_e$. We incorporate it in the definition in order to recover the concept of partial group actions on rings. Indeed, if $G$ is a group with identity element $1_G$ acting partially on a ring $A$ via $\alpha$, then $A_{1_G} = A$ and $\alpha_{1_G} = \text{Id}_A$. 
\item[2.] We omit $\alpha_e = \text{Id}_{A_e}$, for all $e \in \G_0$ from Definition \ref{par-groupoid}, because it is redundant. Indeed, let $e \in \G_0$. Then $e^{-1} = e$ and $\alpha_e^{-1}(A_{e^{-1}}\cap A_e) = A_e$. Given any $x \in A_e$, there exists $y \in A_e$ such that $\alpha_e(y)=x$. Thus, using  (PGr3), it follows that $x = \alpha_e(y) = \alpha_{ee}(y) = \alpha_e \circ \alpha_e(y) = \alpha_e(\alpha_e(y)) = \alpha_e(x)$, concluding that $\alpha_e = \text{Id}_{A_e}$.
\end{itemize}
\end{obs}

Next, we establish a connection between orthogonal partial groupoid actions and partial inverse semigroup with zero actions.

\begin{theorem} \label{corrprimgrport}
    Let $\alpha = (A_s,\alpha_s)_{s \in S}$ be a partial inverse semigroup with zero action of a primitive inverse semigroup $S$ on a commutative ring $A$. Then $\alpha^* = (A_s',\alpha_s')_{s \in S^*}$ is an orthogonal partial groupoid action of the groupoid $S^*$ on the ring $A$, where $A_s' = A_s$ and $\alpha_s' = \alpha_s$, for all $s \in S^*$.

    Conversely, let $\gamma = (A_g,\gamma_g)_{g \in \G}$ be an orthogonal partial groupoid action of the groupoid $\G$ on a commutative ring $A$. Then $\gamma^0 = (A_g',\gamma_g')_{g \in \G^0}$ is a partial inverse semigroup with zero action of the primitive inverse semigroup $\G^0$ on $A$, where $A_g' = A_g$, $\gamma_g' = \gamma_g$, for all $g \in \G$, $A_0' = 0$ and $\gamma_0' = 0$.

    Furthermore, $(\gamma^0)^* = \gamma$ and $(\alpha^*)^0 = \alpha$.
\end{theorem}
\begin{proof}
    For the first part, it is easy to see that $\alpha^*$ satisfies (PGr2) and (PGr3) for all $(s,t) \in (S^*)_2$. Now, since we have that $\alpha_e' = \alpha_e = \text{Id}_{A_e}$, for all $e \in E(S)^* = (S^*)_0$, if we prove that $\alpha^*$ satisfies (PGr0) we also guarantee that $\alpha^*$ satisfies (PGr1). For that, let $e, f \in E(S^*)$. Then $A_e \cap A_f = A_{e \wedge f}$. But $e \wedge f = 0$, since $S$ is primitive. Thus $A_e' \cap A_f' = A_e \cap A_f = A_0 = 0$, for all $e,f \in E(S^*)$ such that $e \neq f$. Since $A = \sum_{e \in E(S)} A_e$, we have that $A = \bigoplus_{e \in E(S)^*} A_e'$.

    Conversely, let us prove that $\gamma^0$ is a partial action of $\G^0$ on $A$. Clearly (PIS1) holds, as
    \begin{align*}
        A = \sum_{e \in \G_0} A_e = \sum_{e \in \G_0} A_e' + 0 = \sum_{e \in \G_0} A_e' + A_0' = \sum_{e \in E(\G^0)} A_e'
    \end{align*}
    and $\gamma_e' = \gamma_e = \text{Id}_{A_e} = \text{Id}_{A_e'}$, for all $e \in E(\G^0)^*$, and $\gamma_0' = 0 = \text{Id}_{A_0'}$.

    Let $g,h \in \G^0$. If $g = 0$, we obtain $\gamma_{h^{-1}}'(A_{g^{-1}}' \cap A_h') = 0 \subseteq A_{(gh)^{-1}}$ and $\gamma_g' \circ \gamma_h'(0) = 0 \gamma_{gh} '(0)$. If $h = 0$, the same affirmations are valid. Assume that $g,h \neq 0$. If $(g,h) \in \G_2$, then the pair $(g,h)$ already satisfies (PIS2) and (PIS3). If $(g,h) \notin \G_2$, then $gh = 0$ in $\G^0$. In this case, $\gamma_{h^{-1}}'(A_{g^{-1}}' \cap A_h') \subseteq \gamma_{h^{-1}}'(A_{d(g)}' \cap A_{r(h)}') = 0 \subseteq A_{(gh)^{-1}}'$ and $\gamma_g' \circ \gamma_h'(0) = 0 = \gamma_{gh}'(0)$.

    The proof of the last assertion is now immediate.
\end{proof}

Assume that $S$ is a 0-$E$-unitary, categorical at zero inverse subsemigroup of $\text{Iso}_{pu}(A)$ for some ring $A$. Let $f \in S$. We know that the $\tau$-class of $f$ coincides with the $\approx$-class of $f$ by Proposition \ref{tauapprox}. In particular, the elements of $\tau(f)$ are strongly compatible. Once $\text{Iso}_{pu}(A)$ is $f$-complete, we can take the element $\alpha_f = \sum_{g \in \tau(f)} g \in \text{Iso}_{pu}(A)$. Let $P' = \{ \alpha_f : f \in S\}$.

\begin{lemma} \label{lemasemzeroparc1}
    On the above notations,

    \begin{enumerate}
        \item[(i)] For every pair $\alpha, \beta \in P'$ such that $\alpha \beta \neq 0$, there is an unique $0 \neq \gamma \in P'$ such that $\alpha \beta \leq \gamma$.

        \item[(ii)] $\alpha_e$ is an idempotent when $e \in E(S)$.

        \item[(iii)] $\alpha_0 = 0$.

        \item[(iv)] If $\alpha \in P'$, then $\alpha^{-1} \in P'$.
    \end{enumerate}
\end{lemma}
\begin{proof}
    (i): By hypothesis, we have $\alpha = \alpha_f$ and $\beta = \alpha_g$, for $f,g \in S$. Since $\alpha \circ \beta \neq 0$, it follows that $f,g \neq 0$. Define $\gamma = \alpha_{fg}$. 
    
    The rest of the proof of (i) and the proof of (ii) and (iv) is now analogous to Lemma \ref{lemaparceunit1}.

    (iii): Clearly $\tau(0) = \{0\}$, from where $\alpha_0 = \sum_{f \in \tau(0)} f = 0$.
\end{proof}

Define the binary operation $\astrosun$ on $P'$ by $\alpha \astrosun \beta = \gamma$, if $\alpha \beta \neq 0$, where $\gamma$ is the unique element of $P'$ that is greater than or equal to $\alpha \circ \beta$, as specified in the previous lemma; and if $\alpha \beta = 0$, define $\alpha \astrosun \beta = 0$.

\begin{prop} \label{lemasemzeroparc2}
    $(P',\astrosun)$ is a primitive inverse semigroup isomorphic to $S/\tau$.
\end{prop}
\begin{proof}
    By Lemma \ref{lemasemzeroparc1}, we have that  $\alpha_f \astrosun \alpha_g = \alpha_{fg}$. To prove that $(P',\astrosun)$ is a primitive inverse semigroup, we only need to show that there is a bijection from $S/\tau$ to $P'$ that preserves binary operations. This is straightforward to check with the map $\tau(f) \mapsto \alpha_f$, which completes the proof.
\end{proof}

With arguments very similar to those used in the previous sections, we can see that there is a global-partial relation between inverse semigroup with zero actions and partial primitive inverse semigroup actions. Hence, by Theorem \ref{corrprimgrport}, also with partial orthogonal groupoid actions. This relationship allows us to define an invariant trace map for global actions of inverse semigroups with zero, constructed similarly to Section 3. Once we have a well defined trace map, we can translate the Galois extensions equivalence theorem from partial orthogonal groupoid actions \cite[Theorem 5.3]{bagio2012partial} to the case of global $0$-$E$-unitary, categorical at zero inverse semigroup actions. Following similar steps from Sections 5 and 6, with minimal adjustments, we achieve a Galois correspondence for finite inverse semigroups categorical at zero. This paper concludes by presenting this theorem. Thus, all the theory developed so far remains applicable in this setting.

\begin{theorem}
    Let $S$ be a finite categorical at zero inverse semigroup that acts on a commutative ring $A$ via an inverse semigroup with zero unital action $\beta = (A_s,\beta_s)_{s \in S}$. Assume that $A$ is $\beta$-Galois over $A^\beta$ and that $A_s \neq 0$, for all $s \in S^*$. Then there is a one-to-one correspondence between the $\beta$-maximal inverse subsemigroups $T$ of $S$ and the $\beta$-strong, separable $A^\beta$-subalgebras of $A$ given by $B \mapsto \beta^{-1}(\beta(S)_B)$ with inverse $T \mapsto A^{\beta|_{T}}$.
\end{theorem}

\end{document}